\documentclass[10pt]{amsart}

\usepackage[utf8]{inputenc}
\usepackage{amsmath,amsthm,amssymb}
\usepackage{indentfirst}
\usepackage{url}
\usepackage{tikz-cd}
\usepackage[colorlinks=true]{hyperref}
\usepackage{colonequals}
\usepackage{graphicx}
\usepackage{enumerate}
\usepackage{soul}
\usepackage{stmaryrd}
\usepackage{mathrsfs,mathtools}
\usepackage{palatino}
\usepackage{bbm}
\usepackage[alphabetic]{amsrefs}

\usepackage[margin=1in]{geometry}

\newtheorem{thm}[subsection]{Theorem}
\newtheorem{lem}[subsection]{Lemma}
\newtheorem{conj}[subsection]{Conjecture}
\newtheorem{prop}[subsection]{Proposition}

\newtheorem{cor}[subsection]{Corollary}
{
\theoremstyle{definition}
\newtheorem{remx}[subsection]{Remark}

\newtheorem{defnx}[subsection]{Definition}
\newtheorem{examplex}[subsection]{Example}
}
\newenvironment{rem}
{\pushQED{\qed}\remx}
{\popQED\endremx}

\newenvironment{defn}
{\pushQED{\qed}\defnx}
{\popQED\enddefnx}
\newenvironment{example}
{\pushQED{\qed}\examplex}
{\popQED\endexamplex}

\newcommand{\NN}{\mathbb N}
\newcommand{\ZZ}{\mathbb Z}

\newcommand{\CC}{\mathbb C}

\newcommand{\OO}{\mathbb O}
\newcommand{\GG}{\mathbb G}

\newcommand{\Aut}{\mathrm{Aut}}

\DeclareMathOperator{\Spec}{Spec}
\DeclareMathOperator{\img}{img}

\DeclareMathOperator{\rk}{rk}
\DeclareMathOperator{\eu}{eu}
\DeclareMathOperator{\codim}{codim}

\newcommand{\fraksl}{\mathfrak{sl}}

\newcommand{\rmSL}{\mathrm{SL}}
\newcommand{\rmGL}{\mathrm{GL}}
\newcommand{\rmPGL}{\mathrm{PGL}}

\renewcommand{\sslash}{/\!\!/}
\newcommand{\ssslash}{/\!\!/\!\!/}

\newcommand{\rs}{\mathrm{rs}}
\newcommand{\Ad}{\mathrm{Ad}}
\newcommand{\diag}{\mathrm{diag}}
\newcommand{\tr}{\mathrm{tr}}

\newcommand{\hhwh}{\mathfrak h\times_{\mathfrak h/W}\mathfrak h}
\newcommand{\wtg}{\widetilde{\mathfrak g}}
\newcommand{\wtnssst}{\widetilde{\mathcal N\ssslash T}}
\newcommand{\stem}{\mathrm{stem}}
\newcommand{\bouq}{\mathrm{bouq}}

\DeclareMathOperator{\Hom}{Hom}
\DeclareMathOperator{\Lie}{Lie}

\begin{document}
\title{Hikita surjectivity for $\mathcal N\ssslash T$}
\author{Linus Setiabrata}
\address{Department of Mathematics, MIT}
\email{setia@mit.edu}

\begin{abstract}
The Hamiltonian reduction $\mathcal N\ssslash T$ of the nilpotent cone in $\mathfrak{sl}_n$ by the torus of diagonal matrices is a Nakajima quiver variety which admits a symplectic resolution $\widetilde{\mathcal N\ssslash T}$, and the corresponding BFN Coulomb branch is the affine closure $\overline{T^*(G/U)}$ of the cotangent bundle of the base affine space. We construct a surjective map $\CC\left[\overline{T^*(G/U)}^{T\times B/U}\right] \twoheadrightarrow H^*\left(\wtnssst\right)$ of graded algebras, which the Hikita conjecture predicts to be an isomorphism. Our map is inherited from a related case of the Hikita conjecture and factors through Kirwan surjectivity for quiver varieties. We conjecture that many other Hikita maps can be inherited from that of a related dual pair.
\end{abstract}
\maketitle
\vspace{-3ex}
\section{Introduction}
Certain conical symplectic singularities $X$ are expected to have a \emph{symplectic dual} $X^!$ which satisfies many striking properties. Although there is no formal definition, nor a systematic procedure to find the dual, in many cases there is a consensus on what the dual ought to be. Examples of symplectic dual pairs include nilpotent orbits in $\fraksl_n$ and Slodowy slices through conjugate orbits, hypertoric varieties and their Gale duals, and Nakajima quiver varieties and BFN Coulomb branches.

Symplectic duality is expected to interchange seemingly unrelated invariants. For example, if $X$ has a symplectic resolution $\widetilde X\to X$, the Hikita conjecture predicts that the cohomology $H^*(\widetilde X)$ is isomorphic to the coordinate ring $\CC[(X^!)^{T^!}]$ of the scheme-theoretic fixed points of $X^!$ with respect to the maximal torus $T^!$ of the group $\Aut_{\mathrm{Pois},\CC^\times}(X^!)$ of Poisson automorphisms of $X^!$ commuting with dilations. When $\widetilde X$ is a Slodowy variety in type $A$, the Hikita conjecture amounts to the theorem \cite{dp81} of deConcini--Procesi expressing the cohomology of a Springer fiber as the coordinate ring of the scheme-theoretic intersection of a nilpotent orbit with the Cartan. Hikita \cite{hikita17} observed that a similar phenomenon holds when $\widetilde X$ is the cotangent bundle of a partial flag variety, the Hilbert scheme of $n$ points in $\CC^2$, and a hypertoric variety. The recent preprint \cite{hkm24} gave a counterexample to the Hikita conjecture (using a conjecturally dual pair from \cite{lmm21}) and proposed a refined version. However, the original version of the Hikita conjecture and various generalizations have been verified for many other dual pairs \cite{ktwwy19,kmp21,ks22,hoang24,shlykov24,chy23}.\\

Throughout, we set $G$ to be $\rmSL_n(\CC)$, $T$ the torus of diagonal matrices, $B$ the Borel subgroup of upper triangular matrices, and $U$ the unipotent subgroup of upper triangular matrices with diagonal entries equal to $1$. The Hamiltonian reduction of the nilpotent cone $\mathcal N\colonequals\mathcal N_{\mathfrak{sl}_n}$ by the maximal torus is defined to be the categorical quotient $\mathcal N\ssslash T\colonequals \{x\in\mathcal N\colon \diag(x) = 0\}\sslash T$ of variety of zero-diagonal nilpotent matrices by the adjoint action of $T$. The variety $\mathcal N\ssslash T$ has a realization as a Nakajima quiver variety for the so-called \emph{bouquet quiver}. It follows that $\mathcal N\ssslash T$ is a conical symplectic singularity (\cite[Thm 1.2]{bs21}).

It is believed \cite[\S 8]{dhk21} that the conical symplectic singularities $\mathcal N\ssslash T$ and $\overline{T^*(G/U)}$ should be symplectic dual. There is now good evidence for this belief: Gannon and Williams \cite{gw23} showed that the Coulomb branch of the $3$-dimensional $\mathcal N = 4$ quiver gauge theory associated to the bouquet is $\overline{T^*(G/U)}$. This places conjectured duality between $\mathcal N\ssslash T$ and $\overline{T^*(G/U)}$ in the larger context of duality between Nakajima quiver varieties and BFN Coulomb branches.

Bellamy \cite{bellamy23} showed that Coulomb branches are conical symplectic singularities: For $\overline{T^*(G/U)}$, this was proven earlier by Jia \cite[Thm 1.1]{jia21}. More generally, it is known (\cite[Thm 1.1]{gannon24}) that $\overline{T^*(G/U)}$ has symplectic singularities in all types, verifying a conjecture of Ginzburg and Kazhdan \cite[Conj 1.3.6]{gk22}.

Bellamy and Schedler \cite{bs21} characterized the Nakajima quiver varieties admitting a symplectic resolution. We show that $\mathcal N\ssslash T$ satisfies their criteria and hence has a symplectic resolution $\wtnssst$. We explicitly construct a variety (Corollary~\ref{cor:generic-deformation}) diffeomorphic to $\wtnssst$. The variety $T^*(G/U)$ has commuting actions of $T$ (induced by left multiplication on $G/U$) and $B/U$ (induced by right multiplication on $G/U$). From these actions one can form the scheme-theoretic fixed points $\overline{T^*(G/U)}^{T\times B/U}$. This scheme is nonreduced and has one closed point, so its coordinate ring is finite-dimensional as a $\CC$-algebra. Furthermore, $\CC[\overline{T^*(G/U)}^{T\times B/U}]$ inherits a $\ZZ$-grading from a conical action of $\CC^\times$ on $\overline{T^*(G/U)}$ commuting with $T\times B/U$ which has multiple explicit descriptions but in particular comes from the description of $\overline{T^*(G/U)}$ as a Coulomb branch.

For the conical symplectic singularity $X = \mathcal N\ssslash T$, the Hikita conjecture predicts that
\begin{equation}
\label{eqn:our-hikita}
H^*\left(\wtnssst\right) \cong \CC\left[\overline{T^*(G/U)}^{T\times B/U}\right]
\end{equation}
as graded rings, with cohomological grading on the left. For $n = 2$ and $n = 3$, the variety $\mathcal N\ssslash T$ is the point and the type $D_4$ Kleinian singularity respectively, while $\overline{T^*(G/U)}$ is the affine space $\CC^4$ and the minimal nilpotent orbit in $\mathfrak{so}_8$ (\cite[Thm 1.3]{jia21}), respectively. The Hikita conjecture is known to hold in these examples (\cite{shlykov24}).

Our main result states: 

\begin{thm}
\label{thm:main}
There is a surjective homomorphism of graded algebras
\[
\CC\left[\overline{T^*(G/U)}^{T\times B/U}\right]\twoheadrightarrow H^*\left(\wtnssst\right).
\]
\end{thm}

The coordinate ring can be realized as a quotient of $\CC[\hhwh]$ in the following way. The scheme $\overline{T^*(G/U)}^{B/U}$ is a subscheme of $\overline{T^*(G/U)}\sslash(B/U) = \mathfrak g\times_{\mathfrak h/W}\mathfrak h$ and hence $\overline{T^*(G/U)}^{T\times B/U}$ is a subscheme of $(\mathfrak g\times_{\mathfrak h/W}\mathfrak h)^T = \hhwh$. We compute the defining ideal of this subscheme: Write $x_1, \dots, x_n,y_1, \dots, y_n$ for coordinates on $\mathfrak h\times\mathfrak h$, and for subsets $S,T\subset[n]\colonequals \{1,\dots,n\}$, set
\[
f_{S,T}\colonequals \prod_{\substack{s\in S\\t\in T}}(x_s - y_t).
\]
\begin{thm}
\label{thm:explicit-generators}
We have
\[
\CC\left[\overline{T^*(G/U)}^{T\times B/U}\right] = \frac{\CC[\hhwh]}{\left\langle f_{S,T}\colon \begin{array}{c}S,T\subseteq[n]\\|S| + |T| = n\end{array}\right\rangle}.
\]
\end{thm}

We show that $\wtnssst$ is diffeomorphic to the quotient of a certain variety $Y_{\lambda,\delta}$ by a free action of the torus $T_{\rmPGL}$ (see \S 3 for an explicit description of $Y_{\lambda,\delta}$; here, $\lambda$ and $\delta$ are sufficiently generic elements of $\CC^n$). From the construction of $Y_{\lambda,\delta}$, there is a natural $T$-equivariant inclusion $Y_{\lambda,\delta}\hookrightarrow\wtg$.

\begin{thm}
\label{thm:surjectivity-and-kernel}
The inclusion $Y_{\lambda,\delta}\hookrightarrow\wtg$ induces a surjection
\[
\CC[\hhwh]\cong H_T^*(\wtg) \twoheadrightarrow H_T^*(Y_{\lambda,\delta})\cong H^*(\wtnssst)
\]
and the kernel contains the functions $\{f_{S,T}\colon |S| + |T| = n\}$.
\end{thm}

Theorems~\ref{thm:explicit-generators} and~\ref{thm:surjectivity-and-kernel} together imply Theorem~\ref{thm:main}.

As $H^*(\wtnssst)$ is a finite-dimensional $\CC$-algebra, the Hikita conjecture~\eqref{eqn:our-hikita} would follow from the equality
\[
\dim_\CC\left(H^*\left(\wtnssst\right)\right) = \dim_\CC\left(\CC\left[\overline{T^*(G/U)}^{T\times B/U}\right]\right).
\]

\subsection*{New cases of the Hikita conjecture from old}
The arguments in this paper suggest that the surjectivity $\CC\left[\overline{T^*(G/U)}^{T\times B/U}\right]\twoheadrightarrow H^*\left(\wtnssst\right)$ can be inherited from the (self-)duality $\mathcal N_{\mathfrak{sl}_n}^! = \mathcal N_{\mathfrak{sl}_n}$. We use in a crucial way (via Kirwan surjectivity for quiver varieties \cite{mn18}) that $\mathcal N\ssslash T$ can be expressed as a quotient of a subvariety of the universal deformation $\wtg$ of $\widetilde{\mathcal N}$ by a free action of a torus. We also use (via the Gelfand--Graev action \cite{gr15,wang21}) that $\overline{T^*(G/U)}$ is the Hamiltonian reduction of a quiver representation space by a product of special linear groups $\mathrm{SL}(V_i)$ and that the corresponding Nakajima quiver variety $\overline{T^*(G/U)}\ssslash (\rmGL(V)/\rmSL(V))$ is the nilpotent cone $\mathcal N$.

More broadly, starting from a pair $(\widetilde G, \mathbf N)$ consisting of a reductive group $\widetilde G$ and a finite dimensional representation $\mathbf N$, one can form the Higgs branch $\mathbf N\ssslash \widetilde G$ and Coulomb branch $\mathcal M_C(\widetilde G,\mathbf N) \colonequals \Spec H^{(\widetilde G)_{\mathcal O}}_*(\mathcal R_{\widetilde G,\mathbf N})$ \cite[Defn 3.13]{bfn-2} of a $3$-dimensional $\mathcal N = 4$ supersymmetric gauge theory. The coordinate ring $\mathcal A(\widetilde G, \mathbf N) \colonequals H_*^{(\widetilde G)_{\mathcal O}}(\mathcal R_{\widetilde G, \mathbf N})$ is equipped with a grading by $\pi_1(\widetilde G)$ (\cite[\S 3(v)]{bfn-2}), so the Pontryagin dual $\pi_1(\widetilde G)^\wedge$ acts on the Coulomb branch $\mathcal M_C(\widetilde G,\mathbf N)$. If $\widetilde G$ acts freely on the fiber $\mu_{\widetilde G}^{-1}(\xi)$ of a generic $\xi\in(\mathrm{Lie}(\widetilde G))^*$ and the quotient $\mu_{\widetilde G}^{-1}(\xi)/\widetilde G$ is smooth, the Hikita conjecture reads
\begin{equation}
\label{eqn:hikita-gen}
H^*(\mu_{\widetilde G}^{-1}(\xi)/\widetilde G) \cong \CC[\mathcal M_C(\widetilde G,\mathbf N)^{\pi_1(\widetilde G)^\wedge}].\tag{$*$}
\end{equation}

Now assume that $\widetilde G$ fits in a short exact sequence $1\to G \to \widetilde G \to T_F \to 1$. On the Higgs branch side, there is an inclusion
\begin{equation}
\label{eqn:gen-coh-inc}
\mu_{\widetilde G}^{-1}(\xi) \hookrightarrow \mu_G^{-1}(\bar\xi),\tag{$\dagger$}
\end{equation}
where $\bar\xi\in(\mathrm{Lie}(G))^*$ is the image of $\xi\in(\mathrm{Lie}(\widetilde G))^*$. The inclusion~\eqref{eqn:gen-coh-inc} induces a map
\begin{equation}
\label{eqn:gen-coh}
H_{T_F}^*(\mu_G^{-1}(\bar\xi)/G) = H_{\widetilde G}^*(\mu_G^{-1}(\bar\xi)) \to H_{\widetilde G}^*(\mu_{\widetilde G}^{-1}(\xi)) = H^*(\mu_{\widetilde G}^{-1}(\xi)/\widetilde G).
\end{equation}

On the Coulomb branch side, let us write $\mathcal A\colonequals \mathcal A(\widetilde G, \mathbf N)$, $K\colonequals \ker(\pi_1(\widetilde G)\twoheadrightarrow\pi_1(T_F))$, $T_F^\vee \colonequals \pi_1(T_F)^\wedge$, and $\mathcal A_K \colonequals \mathcal A_K^{T_F^\vee} = \bigoplus_{\gamma\in K}\mathcal A_\gamma$. Because $\mathcal A_K$ is graded by $K \cong \pi_1(G)$, the space $\Spec(\mathcal A_K) = \mathcal M_C(\widetilde G,\mathbf N)\sslash T_F^\vee$ has an action of $\pi_1(G)^\wedge$. The quotient map
\[
\mathcal A_K \twoheadrightarrow \mathcal A_K/\sum_{\gamma\in\pi_1(T_F)\setminus\{0\}}\mathcal A_\gamma\mathcal A_{-\gamma}
\]
induces a $\pi_1(G)^\wedge$-equivariant inclusion
\begin{equation}
\label{eqn:gen-coord-inc}
\mathcal M_C(\widetilde G,\mathbf N)^{T_F^\vee} \hookrightarrow \mathcal M_C(\widetilde G,\mathbf N)\sslash T_F^\vee \equalscolon \mathcal M_{C,\mathfrak t_F},\tag{$\ddagger$}
\end{equation}
and, by passing to $\pi_1(G)^\wedge$-fixed points, the inclusion~\eqref{eqn:gen-coord-inc} induces a map
\begin{equation}
\label{eqn:gen-coord}
\CC[\mathcal M_{C,\mathfrak t_F}^{\pi_1(G)^\wedge}]\twoheadrightarrow \CC[\mathcal M_C(\widetilde G,\mathbf N)^{\pi_1(\widetilde G)^\wedge}].
\end{equation}
In this setup, the Hikita conjecture~\eqref{eqn:hikita-gen} for the dual pair $(\mathbf N\ssslash \widetilde G, \mathcal M_C(\widetilde G,\mathbf N))$ would follow from:
\begin{itemize}
\item An \emph{equivariant} version of the Hikita conjecture
\begin{equation}
\label{eqn:equiv-hikita-gen}
H_{T_F}^*(\mu_G^{-1}(\bar\xi)/G) \cong \CC[\mathcal M_{C,\mathfrak t_F}^{\pi_1(G)^\wedge}]
\end{equation}
for the dual pair $(\mathbf N\ssslash G,\mathcal M_C(G,\mathbf N))$,
\item Surjectivity of the map~\eqref{eqn:gen-coh}, and
\item An identification of the kernel of~\eqref{eqn:gen-coh} with the kernel of~\eqref{eqn:gen-coord}.
\end{itemize}

In the special case $(X,X^!) = (\mathcal N_{\mathfrak{sl}_n},\mathcal N_{\mathfrak{sl}_n})$ considered in this paper, the first two items are consequences of Borel's presentation for $H_T^*(G/B)$ and Kirwan surjectivity respectively; Theorem~\ref{thm:explicit-generators} gives an explicit description of the kernel of~\eqref{eqn:gen-coord} and Theorem~\ref{thm:surjectivity-and-kernel} implies that the kernel of~\eqref{eqn:gen-coh} contains that of~\eqref{eqn:gen-coord}.

We conjecture that imposing the additional assumption that the Coulomb branch $\mathcal M_C(G,\mathbf N)$ be a Nakajima quiver variety $T^*\mathrm{Rep}(Q^!)\ssslash \mathrm{GL}(V^!)$ may make the conjectured isomorphism~\eqref{eqn:hikita-gen} more tractable:
\begin{conj}
\label{conj:quot-unquot}
Let $X, X^!$ be symplectic dual varieties. Suppose that $X^! = T^*\mathrm{Rep}(Q^!)\ssslash \mathrm{GL}(V^!)$ is a Nakajima quiver variety and write $X^{!,\uparrow}\colonequals T^*\mathrm{Rep}(Q^!)\ssslash \mathrm{SL}(V^!)$. Let $T$ and $T^{!,\uparrow}$ denote the maximal tori of the group of conical Poisson automorphisms of $X$ and $X^{!,\uparrow}$ respectively. Assume that $X\ssslash T$ has a symplectic resolution. Then the Hikita conjecture holds for the $(X\ssslash T, X^{!,\uparrow})$, i.e.,
\[
H^*\left(\widetilde{X\ssslash T}\right) \cong \CC[(X^{!,\uparrow})^{T^{!,\uparrow}}].
\]
\end{conj}
Conjecture~\ref{conj:quot-unquot} is phrased in terms of $X^{!,\uparrow} = T^*\mathrm{Rep}(Q^!)\ssslash \mathrm{SL}(V^!)$ because we expect that the Coulomb branch $\mathcal M_C(T^*\mathrm{Rep}(Q^!),\mathrm{SL}(V^!))$ can be identified with
\[
\mathcal M_C(T^*\mathrm{Rep}(Q^!),\mathrm{GL}(V^!))\ssslash (\mathrm{GL}(V^!)/\mathrm{SL}(V^!)) \cong (\mathbf N\ssslash G)\ssslash T_F = \mathbf N\ssslash\widetilde G,
\]
by \cite[Prop 3.18]{bfn-2}. In good cases (\cite[Thm 1.2.1]{wang21}), it is known that Hamiltonian reductions of the form $T^*\mathrm{Rep}(Q^!)\ssslash \mathrm{SL}(V^!)$ carry an analogue of the Gelfand--Graev action.

In Appendix~\ref{appendix:examples}, we give other examples of symplectic dual pairs which arise as $(X\ssslash T, X^{!,\uparrow})$.

%
We thank Vasily Krylov and an anonymous referee for pointing out many of these connections.

\section*{Acknowledgements}

I am very grateful to my advisor Victor Ginzburg for his guidance, patience, and many fruitful suggestions. I thank Gwyn Bellamy, Vasily Krylov, and Travis Schedler for explaining many details in their work, and I thank Xinchun Ma for many inspiring discussions, including one that led to Conjecture~\ref{conj:quot-unquot}. I thank an anonymous referee for much valuable feedback. I also thank Do Kien Hoang, Boming Jia, Joshua Mundinger, Pavel Shlykov, Minh-T\^am Trinh, Xiangsheng Wang, and Yaochen Wu for helpful correspondence, and I thank Gwyn Bellamy, Vasily Krylov, Joshua Mundinger, and an anonymous referee for useful comments which greatly improved the exposition in this paper. 

\section{{Coordinate ring of }$\overline{T^*(G/U)}^{T\times B/U}$}
\label{sec:coordinate-ring}
The goal of this section is to prove Theorem~\ref{thm:explicit-generators}. To this end, we express the defining ideal of $\overline{T^*(G/U)}^{T\times B/U}$ inside $\hhwh$ in terms of the Gelfand--Graev action (Proposition~\ref{prop:gen-by-spaces}) and uses an explicit description of this action from \cite{wang21} (Lemma~\ref{lem:gg-action}). Only in Proposition~\ref{prop:explicit-generators-prep} and onward do we specialize to type $A$.
\subsection{Scheme-theoretic fixed points}
For an algebraic group $H$ acting on an affine variety $X$, the scheme-theoretic fixed points $X^H$ \cite[VIII,Ex 6.5(d),(e)]{sga3} (see also \cite[Thm 2.3]{fogarty73}) is the (non-reduced, in general) affine scheme defined by
\[
X^H\overset{\rm def}=\Spec\left(\frac{\CC[X]}{\langle f - h(f)\colon f \in \CC[X], h\in H\rangle}\right).
\]
When a torus $T$ acts on $X$, the coordinate ring $\CC[X]$ decomposes into a direct sum of weight spaces $\oplus_{\alpha\in\Lambda_T}\CC[X]_\alpha$, and there is an isomorphism (cf.\ \cite[Prop 1.4]{ks22})
\begin{equation}
\label{eqn:coord-X^T}
\frac{\CC[X]}{\langle f - t(f)\colon f \in \CC[X], t\in T\rangle}\cong \frac{\CC[X]_0}{\sum_{\alpha\neq 0}\CC[X]_\alpha \CC[X]_{-\alpha}}.
\end{equation}

\subsection{Geometry of $T^*(G/U)$ and Gelfand-Graev action}
Write $T^*(G/U)$ for the cotangent bundle of the base affine space $G/U$, and write $\wtg$ for the Grothendieck simultaneous resolution of $\mathfrak g$. Using the Killing form to identify $\mathfrak g\leftrightarrow\mathfrak g^*$, there are identifications $T^*(G/U) \cong G\times_U\mathfrak b \colonequals \frac{G\times\mathfrak b}U$, where $U$ acts on $G\times\mathfrak b$ via $u\cdot(g,b) = (gu,u^{-1}bu)$, and $\wtg\cong G\times_B\mathfrak b\colonequals \frac{G\times\mathfrak b}B$, where $B$ acts on $G\times\mathfrak b$ via $b'\cdot(gb) = (gb, b'^{-1}bb')$. Write $\mathcal B \colonequals G/B$ for the flag variety.

The quotient $T^*(G/U) \to \wtg$ makes $T^*(G/U)$ into a $B/U$-torsor. Write $\Lambda_{B/U}$ for the $(B/U)$-weight lattice. For $\alpha \in \Lambda_{B/U}$, write $\mathcal O_{\wtg}(\alpha)$ for the pullback of the line bundle $\mathcal O_{\mathcal B}(\alpha)$ along $\wtg\to \mathcal B$, so that a regular function on $T^*(G/U)$ of $B/U$-weight $\alpha$ is precisely a section of $\mathcal O_{\wtg}(\alpha)$. As the left $T$-action on $T^*(G/U)$ descends to $\wtg$, we get an isomorphism
\[
\CC[T^*(G/U)] = \bigoplus_{\alpha\in\Lambda_{B/U}}\Gamma(\wtg,\mathcal O_{\wtg}(\alpha))
\]
of $T$-representations.

\begin{lem}[{\cite[Prop 2.6]{broer93}}, {cf.\ also \cite[Lem 3.6.2]{gr15}}]
\label{lem:dominant-surjective}
If $\lambda,\mu\in\Lambda_{B/U}^+$ are dominant weights, the multiplication map $\Gamma(\wtg,\mathcal O_{\wtg}(\lambda))\otimes\Gamma(\wtg,\mathcal O_{\wtg}(\mu)) \to \Gamma(\wtg,\mathcal O_{\wtg}(\lambda+\mu))$ is surjective.
\end{lem}
\begin{proof}
For any dominant weight $\alpha\in\Lambda_{B/U}$, let $V_\alpha\subset\Gamma(\wtg,\mathcal O_{\wtg}(\alpha))$ denote the space of sections obtained by pulling back sections of $\mathcal O_{\mathcal B}(\alpha)$ along the vector bundle map $\wtg\to \mathcal B$. A result of Broer \cite[Prop 2.6]{broer93} implies that when $\alpha$ is dominant the subspace $V_\alpha$ generates $\Gamma(\wtg,\mathcal O_{\wtg}(\alpha))$ as a $\CC[\wtg]$-module.

The lemma follows from the fact that the multiplication map $V_\lambda\otimes V_\mu\to V_{\lambda+\mu}$ is surjective (\cite[Thm 3.1.2(c)]{bk04}.)
\end{proof}

The coordinate ring $\CC[T^*(G/U)]$ is also equipped with an action of $W$, called the \emph{Gelfand-Graev action}, which restricts to isomorphisms
\[
w.\colon \Gamma(\wtg,\mathcal O_{\wtg}(\alpha))\xrightarrow\sim\Gamma(\wtg,\mathcal O_{\wtg}(w.\alpha))
\]
of $T$-representations. Given a weight $\beta\in\Lambda_T$ of $T$, let $\Gamma(\wtg,\mathcal O_{\wtg}(\alpha))_\beta$ denote the $\beta$-weight space of $\Gamma(\wtg,\mathcal O_{\wtg}(\alpha))$. Note that the Gelfand-Graev action preserves $\Gamma(\wtg,\mathcal O_{\wtg}(0))$ as well as its $0$-weight space $\Gamma(\wtg,\mathcal O_{\wtg}(0))_0$.

It is known \cite[Prop 5.5.1]{gr15} 
that the restriction of the Gelfand-Graev action to $\Gamma(\wtg,\mathcal O_{\wtg}(0)) = \CC[\mathfrak g\times_{\mathfrak h/W}\mathfrak h]$ agrees with the one induced by the $W$-action on $\mathfrak g\times_{\mathfrak h/W}\mathfrak h$ given by acting trivially on $\mathfrak g$ and naturally on $\mathfrak h$.

\subsection{Fixed points in $\overline{T^*(G/U)}$}
Consider the algebra
\[
R\colonequals \Gamma(\wtg,\mathcal O_{\wtg}(0))_0
\]
and the ideals
\begin{align*}
I&\colonequals \sum_{(\alpha,\beta)\neq (0,0)}\Gamma(\wtg,\mathcal O_{\wtg}(\alpha))_\beta \cdot \Gamma(\wtg,\mathcal O_{\wtg}(-\alpha))_{-\beta}\\
J&\colonequals \sum_{\beta\neq0}\Gamma(\wtg,\mathcal O_{\wtg}(0))_\beta\cdot\Gamma(\wtg,\mathcal O_{\wtg}(0))_{-\beta}
\end{align*}
of $R$. Equation~\eqref{eqn:coord-X^T} implies that $R/J = \CC[(\mathfrak g\times_{\mathfrak h/W}\mathfrak h)^T] = \CC[\hhwh]$ and that
\begin{align*}
\CC\left[\overline{T^*(G/U)}^{T\times B/U}\right] &= R/I \\&\cong\frac{\CC[\hhwh]}{I/J}.
\end{align*}
By construction, the Gelfand-Graev action on $R$ descends to the $W$-action on $R/J = \CC[\hhwh]$ given by acting on the second $\mathfrak h$ factor.
\begin{prop}
\label{prop:gen-by-spaces}
The ideal $I/J$ of the ring $R/J$ is generated by the image of the $W$-orbit of
\[
\sum_{\omega_i,\lambda}\Gamma(\wtg,\mathcal O_{\wtg}(\omega_i))_\lambda\cdot\Gamma(\wtg,\mathcal O_{\wtg}(-\omega_i))_{-\lambda}
\]
under the projection $R\to R/J$. (The sum runs over fundamental weights $\omega_i\in\Lambda_{B/U}$ and all weights $\lambda\in\Lambda_T$.)
\end{prop}
\begin{proof}
Let $\alpha\in \Lambda_{B/U}$ be a nonzero dominant weight of $B/U$ and let $\beta\in\Lambda_T$ be any weight of $T$. Pick a fundamental weight $\omega_i\in\Lambda_{B/U}$ so that $\alpha-\omega_i$ is dominant. Note that $w_0.(-\alpha)$, $w_0.(-\omega_i)$, and $w_0.(-\alpha+\omega_i)$ are all dominant as well. Lemma~\ref{lem:dominant-surjective} implies that the multiplication maps 
\begin{align*}
\label{eqn:omega-alpha}\sum_{\lambda\in\Lambda_T}\Gamma(\wtg,\mathcal O_{\wtg}(\alpha-\omega_i))_{\beta-\lambda}\otimes\Gamma(\wtg,\mathcal O_{\wtg}(\omega_i))_\lambda &\longrightarrow \Gamma(\wtg,\mathcal O_{\wtg}(\alpha))_\beta\tag{$\dagger$}\\
\label{eqn:omega-negalpha}\sum_{\lambda\in\Lambda_T}\Gamma(\wtg,\mathcal O_{\wtg}(-\alpha+\omega_i))_{-\beta-\lambda}\otimes\Gamma(\wtg,\mathcal O_{\wtg}(-\omega_i))_\lambda&\longrightarrow \Gamma(\wtg,\mathcal O_{\wtg}(-\alpha))_{-\beta}\tag{$\ddagger$}
\end{align*}
are surjective. Taking the tensor product of the maps in~\eqref{eqn:omega-alpha} and~\eqref{eqn:omega-negalpha} and composing with multiplication $\Gamma(\wtg,\mathcal O_{\wtg}(\alpha))_\beta \otimes \Gamma(\wtg,\mathcal O_{\wtg}(-\alpha))_{-\beta} \twoheadrightarrow \Gamma(\wtg,\mathcal O_{\wtg}(\alpha))_\beta\cdot \Gamma(\wtg,\mathcal O_{\wtg}(-\alpha))_{-\beta}$ gives a surjective map
\begin{align*}
\sum_{\lambda,\mu\in\Lambda_T}\underbrace{\Gamma(\wtg,\mathcal O_{\wtg}(\alpha-\omega_i))_{\beta-\lambda}\otimes \Gamma(\wtg,\mathcal O_{\wtg}(-\alpha+\omega_i))_{-\beta-\mu}}_{\subset\CC[\wtg]_{-\lambda-\mu}} \otimes &\underbrace{\Gamma(\wtg,\mathcal O_{\wtg}(\omega_i))_\lambda \otimes \Gamma(\wtg,\mathcal O_{\wtg}(-\omega_i))_\mu}_{\subset\CC[\wtg]_{\lambda+\mu}} \\&\hspace{15ex}\longrightarrow \Gamma(\wtg,\mathcal O_{\wtg}(\alpha))_\beta\cdot \Gamma(\wtg,\mathcal O_{\wtg}(-\alpha))_{-\beta}.
\end{align*}
Under the projection map $I\to I/J$, the images of the subspaces on the left hand side vanish unless $\lambda+\mu = 0$. We deduce that the ideal of $R/J$ generated by the image of
\[
\sum_\lambda\Gamma(\wtg,\mathcal O_{\wtg}(\omega_i))_\lambda \otimes \Gamma(\wtg,\mathcal O_{\wtg}(-\omega_i))_{-\lambda}
\]
contains the image of $\Gamma(\wtg,\mathcal O_{\wtg}(\alpha))_\beta\cdot \Gamma(\wtg,\mathcal O_{\wtg}(-\alpha))_{-\beta}$.

Since $I/J$ is generated by the image of
\[
W\left(\sum_{\substack{\alpha\in\Lambda_{B/U}^+\setminus 0\\\beta\in\Lambda_T}} \Gamma(\wtg,\mathcal O_{\wtg}(\alpha))_\beta\cdot \Gamma(\wtg,\mathcal O_{\wtg}(-\alpha))_{-\beta}\right),
\]
the result follows.
\end{proof}
Recall that $V_\alpha\subset\Gamma(\wtg,\mathcal O_{\wtg}(\alpha))$ denotes the space of sections obtained by pulling back sections of $\mathcal O_{\mathcal B}(\alpha)$ along the vector bundle map $\wtg\to \mathcal B$. The space $V_\alpha$ is stable under the $T$-action. Fix a basis $h_{(\omega_i,\lambda),j}$ for each weight space $(V_{\omega_i})_\lambda$.
\begin{prop}
\label{prop:gen-by-elts}
The $W$-orbits of
\[
h_{(\omega_i,\mu),j} \cdot (w_0.h_{(\omega_{n-i},-\mu),k})
\]
generate $I/J$ as an ideal of $R/J$.
\end{prop}
\begin{proof}
Let
\[
f_1 \in \Gamma(\wtg,\mathcal O_{\wtg}(\omega_i))_\lambda,\qquad f_2 \in\Gamma(\wtg,\mathcal O_{\wtg}(-\omega_i))_{-\lambda}.
\]
Because $V_{\omega_i}$ and $V_{\omega_{n-i}}$ generate $\Gamma(\wtg,\mathcal O_{\wtg}(\omega_i))$ and $w_0.\Gamma(\wtg,\mathcal O_{\wtg}(-\omega_i))$ respectively as a $\CC[\wtg]$-module, we can write
\begin{align*}
f_1 &= \sum_k g_k h_{(\omega_i,\mu_k),k}, \qquad g_k \in \CC[\wtg]_{\lambda-\mu_k}\\
f_2 &= w_0.\left(\sum_k g_k' h_{(\omega_{n-i},\mu_k'),k}\right), \qquad g_k'\in\CC[\wtg]_{-\lambda-\mu_k'}
\end{align*}
and hence
\[
f_1f_2 = \sum_{k,\ell} g_k(w_0.g_\ell') \cdot h_{(\omega_i,\mu_k),k}\cdot (w_0.h_{(\omega_{n-i},\mu_\ell'),\ell}).
\]
Under the projection map $I \to I/J$, the images of the terms on the right hand side vanish unless $\mu_k + \mu_\ell' = 0$. We conclude that the image of $f_1f_2$ in $I/J$ can be written as a $R/J$-linear combination of $h_{(\omega_i,\mu),j} \cdot (w_0.h_{(\omega_{n-i},-\mu),k})$.

The claim now follows from Proposition~\ref{prop:gen-by-spaces}.
\end{proof}
For subsets $S,T\subset[n]$ with $|S| = |T|$, define the function
\begin{align*}
\Delta_{S,T}\colon \rmSL_n &\to \CC\\
g = (g_{ij})_{i,j\in[n]} &\mapsto \det(g_{st})_{s\in S, t\in T}
\end{align*}
The functions $\Delta_S\colonequals \Delta_{S,\{1,\dots,|S|\}}$ are $U$-invariant under right translations and descend to sections of certain line bundles on the flag variety $\mathcal B$: specifically, for each $i$, the set of functions $\{\Delta_S\colon |S| = i\}$ forms a weight basis for the representation $\Gamma(\mathcal B, \mathcal O_{\mathcal B}(\omega_i))$. The $T$-weight of $\Delta_S$ is $\sum_{s\in S}\mu_s$, where $\mu_i\colon T\to\CC^\times$ is the character $\mu_i\colon (t_1,\dots,t_n)\mapsto t_i$.

Let $p$ denote the projection $\wtg = \frac{G\times\mathfrak b}B \to G/B = \mathcal B$ onto the first factor; the map $p$ makes $\wtg$ into a vector bundle over $\mathcal B$.
\begin{prop}
\label{prop:explicit-generators-prep}
Let $p\colon \wtg\to \mathcal B$ denote the vector bundle map. The $W$-orbits of
\[
p^*\Delta_S\cdot(w_0.p^*\Delta_{[n]\setminus S})
\]
generate $I/J$ as an ideal of $R/J = \CC[\hhwh]$.
\end{prop}
\begin{proof}
The map $p^*\colon \Gamma(\mathcal B,\mathcal O_{\mathcal B}(\omega_i)) \xrightarrow\sim V_{\omega_i}$ is an isomorphism of $G$-representations, so $\{p^*\Delta_S\colon |S| = i\}$ forms a weight basis of $V_{\omega_i}$. Proposition~\ref{prop:gen-by-elts} implies that $p^*\Delta_S\cdot(w_0.p^*\Delta_{[n]\setminus S})$ generates $I/J$ as an ideal of $R/J$.
\end{proof}

Finally, we will use an explicit description of the Gelfand-Graev action in type $A$ on the regular semisimple locus. Recall the identification $T^*(G/U)\cong \frac{G\times\mathfrak b}U$. Let $(T^*(G/U))^{\rs}$ denote the image of 
\[
\varphi\colon G\times\mathfrak h^{\rs} \hookrightarrow G\times \mathfrak b \to G\times_U\mathfrak b = T^*(G/U).
\]
\begin{lem}[{\cite[Prop 4.5.1]{wang21}}]
\label{lem:gg-action}
Let $y = \diag(y_1, \dots, y_n) \in \mathfrak h^\rs$. Let $s_k(y)$ denote the $n\times n$ matrix obtained from the identity matrix by replacing the $\{k,k+1\}$-th submatrix by
\[
\begin{bmatrix} 0 & \frac1{y_k - y_{k+1}} \\ y_{k+1} - y_k & 0\end{bmatrix}.
\]
The Gelfand--Graev action of the transposition $s_k$ interchanging $k\longleftrightarrow k+1$ is given by
\begin{align*}
\sigma_k\colon (T^*(G/U))^\rs &\to (T^*(G/U))^\rs\\
[g,y]&\mapsto [gs_k(y), \Ad_{s_k^{-1}}(y)].
\end{align*}
\end{lem}
\begin{lem}
\label{lem:w_0-action}
Let $w_0^{(n)}$ denote the longest element in $S_n$. The action of $w_0^{(n)}$ on $(T^*(G/U))^\rs$ is given by
\[
\left[(g_{ij}), (y_i)\right]\mapsto\left[\left(g_{i,n+1-j}\frac{\prod_{\ell>j}(y_{n+1-j}-y_{n+1-\ell})}{\prod_{\ell<j}(y_{n+1-j}-y_{n+1-\ell})}\right),w_0^{(n)}.(y_i)\right]
\]
\end{lem}
\begin{example}
For $n=4$, the matrix $\left(g_{i,n+1-j}\frac{\prod_{\ell>j}(y_{n+1-j}-y_{n+1-\ell})}{\prod_{\ell<j}(y_{n+1-j}-y_{n+1-\ell})}\right)$ is given by
\[
\begin{bmatrix} g_{14}(y_4 - y_1)(y_4 - y_2)(y_4 - y_3) & g_{13}\frac{(y_3 - y_1)(y_3 - y_2)}{y_3 - y_4} & g_{12} \frac{y_2 - y_1}{(y_2 - y_3)(y_2 - y_4)} & g_{11} \frac1{(y_1 - y_2)(y_1 - y_3)(y_1 - y_4)}\\
g_{24}(y_4 - y_1)(y_4 - y_2)(y_4 - y_3) & g_{23}\frac{(y_3 - y_1)(y_3 - y_2)}{y_3 - y_4} & g_{22} \frac{y_2 - y_1}{(y_2 - y_3)(y_2 - y_4)} & g_{21} \frac1{(y_1 - y_2)(y_1 - y_3)(y_1 - y_4)}\\
g_{34}(y_4 - y_1)(y_4 - y_2)(y_4 - y_3) & g_{33}\frac{(y_3 - y_1)(y_3 - y_2)}{y_3 - y_4} & g_{32} \frac{y_2 - y_1}{(y_2 - y_3)(y_2 - y_4)} & g_{31} \frac1{(y_1 - y_2)(y_1 - y_3)(y_1 - y_4)}\\
g_{44}(y_4 - y_1)(y_4 - y_2)(y_4 - y_3) & g_{43}\frac{(y_3 - y_1)(y_3 - y_2)}{y_3 - y_4} & g_{42} \frac{y_2 - y_1}{(y_2 - y_3)(y_2 - y_4)} & g_{41} \frac1{(y_1 - y_2)(y_1 - y_3)(y_1 - y_4)}\end{bmatrix}
\]
\end{example}
\begin{proof}[Proof of Lemma~\ref{lem:w_0-action}]
We argue by induction. When $n=2$, Lemma~\ref{lem:gg-action} asserts that the action of $w_0^{(2)} = s_1$ is given by
\[
\sigma_1\colon \left[\begin{bmatrix} g_{11} & g_{12} \\ g_{21} & g_{22}\end{bmatrix}, (y_1, y_2)\right]\mapsto \left[\begin{bmatrix} g_{21}(y_2 - y_1) & g_{11} \frac1{y_1 - y_2}\\ g_{22}(y_2 - y_1) & g_{12}\frac1{y_1 - y_2}\end{bmatrix}, (y_2,y_1)\right],
\]
as claimed.

Let $M$ denote the $n\times n$ matrix with
\[
m_{ij} = \begin{cases} g_{i,n-j}\frac{\prod_{\ell>j}(y_{n-j}-y_{n-\ell})}{\prod_{\ell<j}(y_{n-j}-y_{n-\ell})} &\textup{ if } j \leq n-1\\
g_{i,n} &\textup{ if } j=n\end{cases}
\]
Let $\iota\colon S_{n-1}\hookrightarrow S_n$ denote the standard embedding. By induction, $\iota(w_0^{(n-1)})$ acts on $(T^*(\rmSL_n/U))^\rs$ by
\[
[(g_{ij}),(y_i)]\mapsto \left[M, \iota(w_0^{(n-1)}).(y_i)\right].
\]
Using the equality $w_0^{(n)} = s_1\dots s_{n-1}\iota(w_0^{(n-1)})$, repeated application of Lemma~\ref{lem:gg-action} gives the result.
\end{proof}
\newtheorem*{thm:explicit-generators}{Theorem~\ref{thm:explicit-generators}}
\begin{thm:explicit-generators}
We have
\[
\CC\left[\overline{T^*(G/U)}^{T\times B/U}\right] = \frac{\CC[\hhwh]}{\langle f_{S,T}\rangle}.
\]
\end{thm:explicit-generators}

\begin{proof}[Proof of Theorem~\ref{thm:explicit-generators}]
Let $S\subset[n]$ with $|S| = i$. By Lemma~\ref{lem:w_0-action}, the restriction of the function $w_0.p^*\Delta_S$ to $(T^*(G/U))^\rs$ is
\begin{align*}
[g,y]&\mapsto \det\left(g_{s,n-t}\frac{\prod_{\ell>t}(y_{n+1-t}-y_{n+1-\ell})}{\prod_{\ell<t}(y_{n+1-t}-y_{n+1-\ell})}\right)_{\substack{s\in S\\ t\in[i]}}\\&=\det(g_{s,n-t})_{\substack{s\in S\\ t\in[i]}}\frac{\prod_{\substack{a>b\\a>n-|S|}}(y_a - y_b)}{\prod_{\substack{a<b\\a>n-|S|}}(y_a - y_b)}\\&=\varepsilon_S\det(g_{s,n-t})_{\substack{s\in S\\ t\in[i]}}\prod_{\substack{a>n-|S|\\b\leq n-|S|}}(y_a - y_b), \qquad \varepsilon_S\in\{\pm1\}.
\end{align*}
In particular, the function $p^*\Delta_S\cdot(w_0.p^*\Delta_{[n]\setminus S})$ restricts to $(T^*(G/U))^\rs$ as
\[
[g,y]\mapsto \varepsilon_S\det(g_{st})_{\substack{s\in S\\ t\in [i]}}\det(g_{st})_{\substack{s\in[n]\setminus S\\ t\in[n]\setminus [i]}}\prod_{\substack{a> i\\b\leq i}}(y_a - y_b).
\]
Given a permutation $w$, pick a matrix $P_w\in\rmSL_n$ whose entries are nonzero only at positions $(i,w(i))$, and whose nonzero entries are equal to $\pm 1$. 
The function $p^*\Delta_S\cdot(w_0.p^*\Delta_{[n]\setminus S})$ restricts to a nonzero function only when $w(S) =[i]$ and in this case it is given by the formula
\[
[P_w,y]\mapsto \varepsilon_S'\prod_{\substack{a>i\\b\leq i}}(y_a - y_b), \qquad \varepsilon_S'\in\{\pm 1\}.
\]
It follows that the image of $p^*\Delta_S\cdot(w_0.p^*\Delta_{[n]\setminus S})\in R = \CC[T^*(G/U)]_{(0,0)}$ under the projection $R \to R/J = \CC[\hhwh]$ satisfies the formula
\begin{equation}
\label{eqn:hreg}
(w.y,y)\mapsto\begin{cases}\varepsilon_S'\prod_{s\in[n]\setminus S} (y_{w(s)}-y_1)\dots(y_{w(s)}-y_i) &\textup{ if } w(S) = [i]\\ 0&\textup{ else}\end{cases}
\end{equation}
for any $y\in\mathfrak h^{\mathrm{reg}}$. But the restriction of
\[
\varepsilon_S' f_{[n]\setminus S,[i]} = \varepsilon_S' \prod_{\substack{s\in[n]\setminus S\\t \in[i]}}(x_s-y_t)
\]
to $\hhwh$ also satisfies Equation~\eqref{eqn:hreg} for all $y\in\mathfrak h^{\mathrm{reg}}$, and so the images of $p^*\Delta_S\cdot(w_0.p^*\Delta_{[n]\setminus S})$ and $\varepsilon_S' f_{[n]\setminus S,[i]}$ are equal in $\CC[\hhwh]$.

The theorem follows from the fact that the $W$-orbit of $f_{[n]\setminus S, [i]}$ is $\{f_{[n]\setminus S, T}\colon |T| = i\}$.
\end{proof}
\begin{rem}
The set $\{f_{S,T}\colon |S| + |T| = n\}$ is not a minimal generating set of the ideal $\langle f_{S,T}\rangle$; for example one can compute that
\[
f_{S,T} = (-1)^{|S||T|} f_{[n]\setminus S, [n]\setminus T}
\]
as functions on $\CC[\hhwh]$.
\end{rem}

\section{{Symplectic resolution of }$\mathcal N\ssslash T$}
\label{sec:symplectic-resolution}
In this section we prove that $\mathcal N\ssslash T$ is a Nakajima quiver variety (Proposition~\ref{prop:nssst-is-nqv}) with a symplectic resolution $\wtnssst$ (Theorem~\ref{thm:symplectic-resolution}). Using a standard construction for Nakajima quiver varieties in general, we give an explicit description of the diffeomorphism type of $\wtnssst$ in Corollary~\ref{cor:generic-deformation}: the variety $\wtnssst$ is diffeomorphic to the quotient $Y_{\lambda,\delta}/(\CC^\times)^{n-1}$ for sufficiently generic $\lambda,\delta\in \{(x_1, \dots, x_n)\in \CC^n\colon x_1 + \dots + x_n = 0\}$, where
\[
Y_{\lambda,\delta}\colonequals \left\{(x,F_\bullet) \in \wtg\colon \begin{array}{c} x|_{F_k/F_{k-1}} = \lambda_k\mathrm{Id}\\ \diag(x) = \delta\end{array}\right\},
\]
and the torus $(\CC^\times)^{n-1}$ acts on $Y_{\lambda,\delta}$ by
\[
(t_1, \dots, t_{n-1}) \cdot (x,F_\bullet) = (\Ad_{t'}(x), t'\cdot F_\bullet), \qquad t' \colonequals \diag(t_1, \dots, t_{n-1}, 1).
\]
(Above, for $x\in\mathfrak g$ we write $\diag(x) \colonequals (x_{11}, \dots, x_{nn})$ for the vector of diagonal entries of $x$.)
\subsection{Nakajima quiver varieties}
\label{subsec:nakajima-quiver-varieties}
Let $Q = (I,E)$ be a quiver, and fix dimension vectors $\mathbf v, \mathbf w \in \ZZ_{\geq 0}^I$. Let $s,t\colon E\to I$ denote the source and target maps respectively. Write
\begin{align*}
\mathbb M(Q,\mathbf v,\mathbf w) &\colonequals \bigoplus_{e \in E}\Hom(\CC^{v_{s(e)}},\CC^{v_{t(e)}}) \oplus \bigoplus_{e\in E}\Hom(\CC^{v_{t(e)}}, \CC^{v_{s(e)}}) \oplus\bigoplus_{i\in I}\Hom(\CC^{w_i}, \CC^{v_i}) \oplus\bigoplus_{i\in I}\Hom(\CC^{v_i},\CC^{w_i})
\\
\label{eqn:rep-space}&\cong T^*\left(\bigoplus_{e\in E}\Hom(\CC^{v_{s(e)}},\CC^{v_{t(e)}}) \oplus\bigoplus_{i\in I}\Hom(\CC^{w_i}, \CC^{v_i})\right).\tag{$\heartsuit$}
\end{align*}
The vector space $\mathbb M$ is equipped with a canonical symplectic form and an action of the group
\[
\GG_{\mathbf v} \colonequals \prod_{i\in I}\rmGL(v_i)
\]
induced by the action of $\rmGL(v_i)$ on $\CC^{v_i}$.

Write $\mathfrak g_{\mathbf v}\colonequals \Lie(\GG_{\mathbf v})$ and let $\mu\colon \mathbb M \to \mathfrak g_{\mathbf v}^*$ denote the moment map of this action. 

The Nakajima quiver variety is the GIT quotient
\[
\mathfrak M_{\lambda,\theta}(\mathbf v, \mathbf w)\colonequals \mu^{-1}(\lambda)\sslash_{\!\theta\,} \GG_{\mathbf v}
\]
where $\lambda \in Z(\mathfrak g_{\mathbf v}^*)$ and $\theta = (\theta_i)_{i\in I}\in\ZZ^I$ is a stability condition, encoding a character $\chi_\theta\colon\mathbb G_{\mathbf v}\to\CC^\times$ via $\chi_\theta((g_i)_{i\in I}) \colonequals \prod_{i\in I}\det(g_i)^{\theta_i}$.

Given dimension vectors $\mathbf v = (v_i)_{i\in I}$ and $\mathbf v' = (v_i')_{i\in I}$, write $\mathbf v' < \mathbf v$ if $v_i' \leq v_i$ for all $i$ and $\mathbf v' \neq \mathbf v$.

\begin{defn}[{\cite[Defn 3.1]{mn18}}]
The stability condition $\theta$ is \emph{nondegenerate} if $\sum_i \theta_i \cdot v_i' \neq 0$ for all nonzero dimension vectors $\mathbf v' < \mathbf v$.
\end{defn}

In the special case $\mathbf w = \mathbf e_j$ (i.e., exactly one vertex is framed and the framing is one-dimensional), and the stability parameter $\theta = (\theta_i)_{i\in I}$ is in $\ZZ_{>0}^I$ (hence is nondegenerate), the locus of $\theta$-stable points is particularly easy to compute using King stability conditions \cite[Prop 3.1]{king94}.

\begin{lem}[{\cite[Lem 3.8, Lem 3.10]{nakajima98}}]
\label{lem:stability-equivalence}
Let $\mathbf w = \mathbf e_j$ and $\theta \in \ZZ_{>0}^I$. Given a point $p \in \mathbb M(Q, \mathbf v, \mathbf w)$, write $j(p)$ for the framing map $\CC^{v_i} \to \CC$. The following are equivalent:
\begin{enumerate}
\item $p$ is $\theta$-stable
\item $p$ is $\theta$-semistable
\item $\ker(j(p))$ contains no nonzero $p$-stable $I$-graded subspace.
\end{enumerate}
The action of $\GG_{\mathbf v}$ is free on the stable locus $\mathbb M(Q,\mathbf v, \mathbf w)^{\theta\textup{-s}}$ and consequently
\[
\mathfrak M_{\lambda,\theta}(\mathbf v, \mathbf w) = (\mu^{-1}(\lambda)\cap\mathbb M^{\theta\textup{-s}})/\GG_{\mathbf v}.
\]
\end{lem}

The inclusion $\mu^{-1}(0)^{\theta\textup{-s}}\hookrightarrow \mu^{-1}(0)$ induces a map in equivariant cohomology
\[
H_{\GG_{\mathbf v}}^*(\textup{pt}) \cong H_{\GG_{\mathbf v}}^*(\mu^{-1}(0)) \to H_{\GG_{\mathbf v}}^*(\mu^{-1}(0)^{\theta\textup{-s}}) \cong H^*(\mathfrak M_{0,\theta}(\mathbf v, \mathbf w))
\]
called the \emph{Kirwan map}. (The last isomorphism follows from Lemma~\ref{lem:stability-equivalence}.)

A special case of \emph{Kirwan surjectivity}, due to McGerty--Nevins, reads:
\begin{thm}[{\cite[Thm 1.2]{mn18}}]
\label{thm:kirwan-surjectivity}
Let $\mathfrak M_{0,\theta}(\mathbf v, \mathbf w)$ be a smooth Nakajima quiver variety. Then the Kirwan map is surjective.
\end{thm}

We say that a Nakajima quiver variety is \emph{unframed} if $\mathbf w = 0$. When the dimension vector $\mathbf v$ of an unframed Nakajima quiver variety is in a certain combinatorially defined set $\Sigma_0(Q)$, Bellamy and Schedler provide a criterion for $\mathfrak M$ to admit a projective symplectic resolution given by deforming the stability parameter.

Following the notation in \cite{bs21}, the Ringel form on $\ZZ^{Q_0}$ is defined by
\[
\langle \alpha,\beta\rangle \colonequals \sum_{i\in Q_0}\alpha_i\beta_i - \sum_{\alpha\in Q_1}\alpha_{t(a)}\beta_{h(a)}.
\]
\begin{defn}[{\cite[\S 2.2]{bs21}}]
The vector $\hat{\mathbf v}$ is \emph{anisotropic} if $\langle \alpha,\alpha\rangle < 0$.
\end{defn}

\begin{prop}[{\cite[Thm 1.5]{bs21}}]
\label{prop:bs-criterion}
Let $\mathbf v \in \Sigma_0(Q)$. Then the Nakajima quiver variety $\mathfrak M_{0,0}(\mathbf v,\mathbf 0)$ admits a projective symplectic resolution if $\mathbf v$ is indivisible. If $\mathbf v$ is anisotropic, a resolution is given by moving to a generic stability parameter.
\end{prop}

\subsection{Presentation of $\mathcal N\ssslash T$ as a Nakajima quiver variety}
The \emph{(trimmed) bouquet quiver} $Q_n$ is the quiver with vertices $\{s_1, \dots, s_{n-1}\}\sqcup\{b_1, \dots, b_{n-1}\}$ and edges
\[
\{(s_i, s_{i+1})\colon i \in [n-1]\} \sqcup \{(s_{n-1}, b_i)\colon i\in[n-1]\}.
\]
See Figure~\ref{fig:Q4} for an example. We write dimension vectors $\mathbf v$ for $Q_n$ as tuples $(v_{s_1}, \dots, v_{s_{n-1}} ; v_{b_1}, \dots, v_{b_{n-1}})$. The vertices $s_i$ and $b_i$ are called \emph{stem} and \emph{bouquet} vertices. Associated to this partition of the vertices of $Q_n$, we write 
\[
\GG_\stem\colonequals \prod_{\substack{i \in I(Q_n)\\ i = s_j}}\rmGL(v_i), \qquad \GG_\bouq\colonequals \prod_{\substack{i\in I(Q_n)\\ i = b_j}}\rmGL(v_i),
\]
and $\mu_\stem, \mu_\bouq$ for the corresponding moment maps.

The \emph{abundant bouquet quiver} $Q_n^+$ has an extra bouquet vertex: it is the quiver with vertices $\{s_1, \dots, s_{n-1}\}\sqcup\{b_1, \dots, b_n\}$ and edges
\[
\{(s_i, s_{i+1})\colon i \in [n-1]\} \sqcup \{(s_{n-1}, b_i)\colon i\in[n]\}.
\]
See Figure~\ref{fig:Q4} for an example. We write dimension vectors $\hat{\mathbf v}$ for $Q_n^+$ as tuples $(v_{s_1}, \dots, v_{s_{n-1}}; v_{b_1}, \dots, v_{b_n})$.

\begin{figure}[ht]
\includegraphics[scale=0.9]{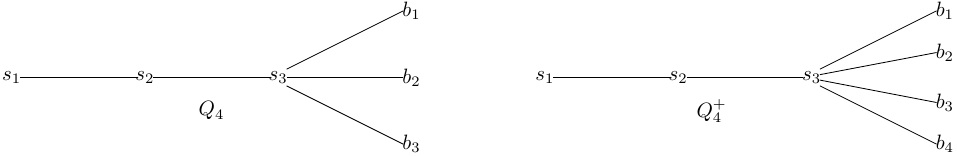}
\caption{Left: The bouquet quiver $Q_4$, with vertices labelled. Right: The abundant bouquet quiver $Q_4^+$, with vertices labelled.}
\label{fig:Q4}
\end{figure}

A special case of the Crawley-Boevey trick \cite[pg.\ 261]{cb01} says that the Nakajima quiver variety $\mathfrak M_{\lambda,\theta}(\mathbf v, \mathbf e_{s_{n-1}})$ for the trimmed bouquet $Q_n$ is isomorphic to the unframed Nakajima quiver variety $\mathfrak M_{\hat\lambda,\hat\theta}(\hat{\mathbf v}, \mathbf 0)$ for the abundant bouquet quiver $Q_n^+$, where:
\begin{align*}
\hat{\mathbf v}&\colonequals (\mathbf v, 1),\\
\hat\lambda &\colonequals \left(\lambda, -\sum_{i\in V(Q_n)} \lambda_i v_i\right ),\\
\hat\theta &\colonequals \left(\theta, -\sum_{i\in V(Q_n)}\theta_i v_i\right ).
\end{align*}

\begin{prop}
\label{prop:nssst-is-nqv}
The variety $\mathcal N\ssslash T$ is isomorphic to the unframed Nakajima quiver variety $\mathfrak M_{0,0}(\hat{\mathbf v}, \mathbf 0)$ for the quiver $Q_n^+$ and dimension vector $\hat{\mathbf v} = (1, \dots, n-1; 1, \dots, 1)$.
\end{prop}
\begin{proof}
Observe that the restriction of $Q_n$ to the stem vertices is a type $A_{n-1}$ Dynkin quiver, and write $\hat{\mathbf v}|_\stem \colonequals (1, \dots, n-1)$ for the restriction of $\hat{\mathbf v}$ to the stem vertices. The Hamiltonian reduction $\mathbb M(Q_n^+,\hat{\mathbf v}, \mathbf 0)\ssslash \GG_\stem\colonequals \mu_\stem^{-1}(0)\sslash\GG_\stem$ is the Nakajima quiver variety 
\[
\mathfrak M_{0,0}(\hat{\mathbf v}|_{\stem}, n\,\mathbf e_{s_{n-1}}),
\]
which is known \cite[Thm 7.2]{nakajima94} to be nilpotent cone $\mathcal N$. The map $\mu_\bouq$ descends to $\mathbb M(Q_n^+,\hat{\mathbf v}, \mathbf 0)\ssslash \GG_\stem \cong \mathcal N$ and coincides with the map sending a nilpotent matrix $x = (x_{ij})_{i,j\in[n]}$ to the tuple $(x_{11}, \dots, x_{n-1,n-1})$ of diagonal entries of $x$, and the residual action of $\GG_\bouq$ coincides with the adjoint action of the torus $T'\colonequals \{\diag(t_1, \dots, t_{n-1}, 1)\}$ on $\mathcal N$.

By Hamiltonian reduction in stages (\cite[Thm 5.2.9]{mmopr07}, see also \cite[Thm 3.3.1]{morgan14}),
\begin{align*}
\mathfrak M_{0,0}(\hat{\mathbf v}, \mathbf 0) &\cong (\mathbb M(Q_n^+, \hat{\mathbf v}, \mathbf 0)\ssslash \GG_\stem)\ssslash \GG_\bouq\\ &\cong \mathcal N\ssslash \GG_\bouq\\&= \mu_\bouq^{-1}(0)\sslash T'.
\end{align*}
As nilpotent matrices are traceless we have $\mu_\bouq^{-1}(0) = \{x\in\mathcal N\colon \diag(x) = 0\}$. Furthermore, the composite $T'\hookrightarrow T_{\mathrm{GL}_n} \twoheadrightarrow T_{\mathrm{PGL}_n}$ is an isomorphism, and the adjoint action of $T$ factors through $T\to T_{\mathrm{PGL}_n}$; it follows that
\[
\mu_\bouq^{-1}(0)\sslash T' \cong \{x\in\mathcal N\colon \diag(x) = 0\}\sslash T.\qedhere
\]
\end{proof}
\begin{rem}
By the Crawley-Boevey trick, $\mathcal N\ssslash T$ is also isomorphic to the Nakajima quiver variety $\mathfrak M_{0,0}(\mathbf v, \mathbf e_{s_{n-1}})$ for the quiver $Q_n$ and dimension vector $\mathbf v = (1,2,\dots,n-1;1,\dots,1)$.
\end{rem}

\begin{thm}
\label{thm:symplectic-resolution}
The variety $\mathcal N\ssslash T$ has a symplectic resolution $\wtnssst$ given by moving to a generic stability parameter.
\end{thm}
\begin{proof}
A routine computation (Lemma~\ref{lem:v-in-sigma}) shows that the dimension vector $\hat{\mathbf v}$ is in the set $\Sigma_0(Q_n^+)$. Lemma~\ref{lem:v-anisotropic} guarantees that $\hat{\mathbf v}$ is anisotropic when $n\geq 4$. The claim follows from Proposition~\ref{prop:bs-criterion}.

For $n = 3$, the quiver $Q_n^+$ is the affine $\widetilde D_4$ Dynkin quiver, and the dimension vector $\hat{\mathbf v}$ is the minimal imaginary root; thus $\mathcal N\ssslash T$ is the Kleinian singularity of type $D_4$. The claim is then a special case of Kronheimer's construction \cite{kronheimer89} of the minimal resolution via quiver varieties.
\end{proof}

As the diffeomorphism type of a smooth Nakajima variety does not depend on the choice of (generic) moment map and stability parameter (\cite[Cor 4.2]{nakajima94}), Theorem~\ref{thm:symplectic-resolution} implies the following claim.

\begin{prop}[{cf.\ \cite[Cor 4.2]{nakajima94}}]
\label{prop:diffeo-type}
Let $\theta\colonequals (1, \dots, 1)$. For sufficiently generic $(\nu;\gamma)\in Z(\mathfrak g_{\mathbf v}^*)$, the variety $\wtnssst$ is diffeomorphic to $\mathfrak M_{(\nu;\gamma),\theta}(\mathbf v, \mathbf e_{s_{n-1}})$.
\end{prop}

\subsection{Stable loci conditions}

Our next goal is to use King stability conditions to compute stable loci explicitly (Lemma~\ref{lem:theta-stable-description}) in order to give an explicit description of the diffeomorphism type of $\wtnssst$.

As in the previous subsection, let $\theta\colonequals (1,\dots,1)$. We use the notation of Subsection~\ref{subsec:nakajima-quiver-varieties} for the quiver $Q_n$. Viewing $\mathcal N\ssslash T$ and $\wtnssst$ as quiver varieties for $Q_n$, an element $(\mathbf x, \mathbf y, \alpha,\beta)$ of the corresponding representation space $\mathbb M(Q_n, \mathbf v, \mathbf e_{s_{n-1}})$ consists of maps $x_i\colon \CC^i \to \CC^{i+1}$, $y_i\colon \CC^{i+1}\to\CC^i$ between stem vertices, along with $n$ maps $\alpha_i\colon \CC^{n-1}\to\CC$ and $n$ maps $\beta_i\colon \CC\to\CC^{n-1}$, where for $i\in[n-1]$ the maps $\alpha_i, \beta_i$ are incident to $s_{n-1}$ and $b_i$ and $\alpha_n,\beta_n$ are framing maps on $s_{n-1}$; see Figure~\ref{fig:Q4-maps}.
\begin{figure}[ht]
\includegraphics{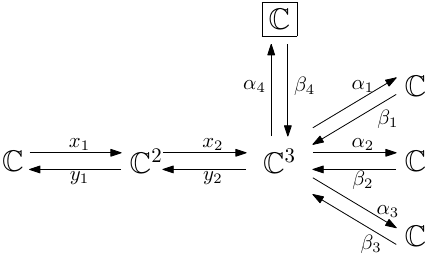}
\caption{An element of the representation space $\mathbb M(Q_4,\mathbf v,\mathbf e_{s_3})$. (The boxed $\CC$ is the framing associated with the vertex $s_3$.)}
\label{fig:Q4-maps}
\end{figure}

For $(\mathbf x, \mathbf y, \alpha,\beta)$ in $\mathbb M$, it will be useful to combine the $\alpha$ and $\beta$ into the maps
\begin{align*}
\varphi\colon \CC^{n-1}&\to\CC^n\\
v&\mapsto(\alpha_1(v),\dots,\alpha_n(v))
\end{align*}
and
\begin{align*}
\psi\colon\CC^n&\to\CC^{n-1}\\
(v_1,\dots,v_n)&\mapsto\beta_1(v_1) + \dots + \beta_n(v_n).
\end{align*}
The map $\varphi\psi\colon \CC^n\to\CC^n$ is given by the matrix $M$ whose $ij$-th entry is $M_{ij} = \alpha_j\beta_i(1)$.

\begin{lem}
\label{lem:theta-stable-description}
We have
\[
\mathbb M^{\theta\textup{-s}} = \left\{(\mathbf x, \mathbf y, \alpha,\beta)\colon \begin{array}{c} x_i, \varphi \textup{ injective}\\\varphi\psi \textup{ satisfies~\eqref{eqn:key-stability}}\end{array}\right\},
\]
where the key stability condition is given by:
\begin{equation}
\label{eqn:key-stability}\tag{$*$}
x\colon\CC^n\to\CC^n \textup{ does not preserve any nonzero coordinate subspace of }\CC^{n-1}\subset\CC^n.
\end{equation}
\end{lem}
\begin{proof}
The condition that $\varphi\psi$ satisfies~\eqref{eqn:key-stability} is equivalent to:
\begin{equation}
\label{eqn:key-stability-real}\tag{$**$}
\textup{For all nonempty $S\subseteq[n-1]$, there exists $i\in S$ such that $\img\beta_i\not\subset\ker\alpha_j$ for some $j\not\in S$.}
\end{equation}
(Above, $j$ may be equal to $n$.) In other words, there may not exist a subset $S\subseteq[n-1]$ where $\img\alpha_i\subset\ker\beta_j$ whenever $i\in S$ and $j\not\in S$: such an $S$ exists if and only if $\varphi\psi$ fixes the coordinate subspace $\CC^S\subset\CC^{n-1}$.

By Lemma~\ref{lem:stability-equivalence}, the point $(\mathbf x,\mathbf y,\alpha,\beta)$ is $\theta$-stable if and only if the only $(\mathbf x,\mathbf y,\alpha,\beta)$-stable subspace of $\ker\alpha_n$ is zero.

We first claim that if any $x_k$ has nonzero kernel, the subspace
\[
\ker x_k \oplus \bigoplus_{i=1}^{k-1} y_iy_{i+1}\dots y_{k-1}(\ker x_k) 
\]
of $\ker\alpha_n$ is $(\mathbf x,\mathbf y,\alpha,\beta)$-stable, as
\begin{align*}
x_iy_iy_{i+1}\dots y_{k-1}(\ker x_k) &= y_{i+1}x_{i+1}y_{i+1}\dots y_{k-1}(\ker x_k) \\&= \dots \\&= y_{i+1}y_{i+2}\dots y_{k-1}x_{k-1}y_{k-1}(\ker x_k)\\&= y_{i+1}y_{i+2}\dots y_{k-1}y_k x_k(\ker x_k) \\&= 0.
\end{align*}
Similarly, if $\varphi$ has nonzero kernel, then the subspace
\[
\ker\varphi\oplus \bigoplus_{i=1}^{n-2} y_iy_{i+1}\dots y_{n-2} \ker\varphi
\]
of $\ker\alpha_n$ is $(\mathbf x,\mathbf y,\alpha,\beta)$-stable: as above, one iteratively applies the moment map equations $x_jy_j = y_{j+1}x_{j+1}$ when $j\leq n-3$, and then applies the moment map condition $x_{n-2}y_{n-2} = \psi\varphi$.

Next, suppose that there is a nonzero coordinate subspace $\CC^S\subset\CC^{n-1}$ fixed by $\varphi\psi$. Identifying $\CC^S$ with the direct sum of the one-dimensional vector spaces $\CC$ at the vertices $\{b_i\colon i\in S\}$, we claim the subspace
\[
\CC^S \oplus \psi(\CC^S) \oplus \bigoplus_{i=1}^{n-2}y_iy_{i+1}\dots y_{n-2}\psi(\CC^S) 
\]
of $\ker\alpha_n$ is $(\mathbf x,\mathbf y,\alpha,\beta)$-stable: applying moment map equations we compute that
\begin{align*}
x_i(y_iy_{i+1}\dots y_{n-2}\psi(\CC^S)) &= y_{i+1}\dots y_{n-2}\psi\varphi\psi(\CC^S) \\&= y_{i+1}\dots y_{n-2}\psi(\CC^S)
\end{align*}
because $\varphi\psi$ fixes $\CC^S$; furthermore, the equality the equality $\alpha_i = \pi_i\varphi$ (with $\pi_i\colon \CC^n\to\CC$ the projection to the $i$-th coordinate) implies
\[
\alpha_i\psi(\CC^S) \subseteq \begin{cases} \CC &\textup{ if } i \in S\\ 0 &\textup{ else} \end{cases}
\]
because $\varphi\psi$ fixes $\CC^S$, while $\beta_i(\CC) = \psi(\CC^{\{i\}})$ implies
\[
\beta_i(\CC) \subset \psi(\CC^S) \qquad\textup{ when } i \in S.
\]
We deduce that
\[
\mathbb M^{\theta\textup{-s}}\subseteq\left\{(\mathbf x, \mathbf y, \alpha,\beta)\colon \begin{array}{c} x_i, \varphi \textup{ injective}\\\varphi\psi \textup{ satisfies~\eqref{eqn:key-stability}}\end{array}\right\}.
\]
We now show that if $x_i, \varphi$ are injective and $\varphi\psi$ satisfies~\eqref{eqn:key-stability} then $(\mathbf x,\mathbf y,\alpha,\beta)$ is $\theta$-stable. Let $V = \bigoplus_{i\in[n-1]} V_{s_i} \oplus \bigoplus_{i\in[n-1]}V_{b_i}$ be an $(\mathbf x,\mathbf y,\alpha,\beta)$-stable subspace of $\ker\alpha_n$. Let $S$ denote the set of $i\in[n-1]$ for which $V_{b_i}$ is nonzero. Since $V$ is $(\alpha,\beta)$-stable, we know $\img\beta_i\subset\ker\alpha_j$ whenever $i \in S$ and $j\not\in S$ (cf.~\eqref{eqn:key-stability-real}); thus $S$ must be empty. Since $\varphi$ is injective and $V\subset\ker\alpha_n$, we deduce $V_{s_{n-1}} = 0$ as well. Finally, since the $x_i$ are injective we conclude that $V_{s_i} = 0$ for $i \in [n-1]$.
\end{proof}

\begin{lem}
\label{lem:tilde-g}
The group $\GG_\stem$ acts freely on $\{(\mathbf x, \mathbf y, \alpha, \beta)\in\mu^{-1}(Z(\mathfrak g_{\mathbf v}^*))\colon x_i, \varphi\textup{ injective}\}$, and there is an isomorphism
\begin{align*}
\{(\mathbf x, \mathbf y, \alpha, \beta)\in\mu^{-1}(Z(\mathfrak g_{\mathbf v}^*))\colon x_i, \varphi\textup{ injective}\}/\GG_\stem &\xrightarrow{\,\sim\,} \wtg\\
[(\mathbf x, \mathbf y, \alpha,\beta)] &\longmapsto \left(\varphi\psi - \frac1n \tr(\varphi\psi)\cdot\mathrm{Id}, F_\bullet\right),
\end{align*}
where $F_\bullet$ is the flag defined by $F_k = \img(\varphi x_{n-2}\dots x_k)$. Furthermore,
\begin{equation}
\label{eqn:M-eval}
(\varphi\psi)|_{F_k/F_{k-1}} = \nu_{n-1} + \dots + \nu_k,
\end{equation}
where $\nu = (\nu_1, \dots, \nu_{n-1}) \colonequals \mu_\stem(\mathbf x, \mathbf y, \alpha,\beta)$. (In particular, $(\varphi\psi)|_{F_n/F_{n-1}} = 0$.)
\end{lem}
\begin{proof}

Recall that the restriction of $Q_n$ to the stem vertices is the type $A_{n-1}$ Dynkin quiver, and observe that $\mathbb M(Q_n, \mathbf v, \mathbf e_{s_{n-1}})$ and $\mathbb M(Q_n|_\stem, \mathbf v|_\stem, n\mathbf e_{s_{n-1}})$ are isomorphic as vector spaces via the map $(\mathbf x, \mathbf y, \alpha,\beta) \mapsto (\mathbf x, \mathbf y, \varphi,\psi)$.

It is well known that
\[
\{(\mathbf x, \mathbf y, \alpha, \beta)\in\mu^{-1}(Z(\mathfrak g_{\mathbf v}^*))\colon x_i, \varphi\textup{ injective}\} = \mu_\stem^{-1}(Z(\mathfrak g_\stem^*))^{\theta'\textup{-s}},
\]
where $\theta' \colonequals (1,\dots,1)$ is a stability condition for $Q_n|_\stem$. As $\mu_\stem$ is the moment map for $Q_n|_\stem$, Lemma~\ref{lem:stability-equivalence} implies that $\GG_\stem$ acts freely on $\{(\mathbf x, \mathbf y, \alpha, \beta)\in\mu^{-1}(Z(\mathfrak g_{\mathbf v}^*))\colon x_i, \varphi\textup{ injective}\}$. Furthermore, \cite[Thm 7.18]{dks13} (see also \cite[(4.5)]{wang21}) implies that there is an isomorphism
\begin{align*}
(\mu_\stem)^{-1}(Z(\mathfrak g_\stem^*))^{\theta'\textup{-s}}/\GG_\stem &\xrightarrow{\,\sim\,}\wtg\\
[(\mathbf x, \mathbf y, \psi,\varphi)]&\longmapsto \left(\varphi\psi - \frac1n\tr(\varphi\psi)\cdot\mathrm{Id},F_\bullet\right).
\end{align*}
It is left to prove Equation~\eqref{eqn:M-eval}. To this end, fix any $v \in F_k$, so that $v = \varphi x_{n-2}\dots x_k(w)$ for some $w$. The moment map equations give
\begin{align*}
\varphi\psi(v) &= \varphi\psi\varphi x_{n-2}\dots x_k(w)\\
&= \varphi\circ[x_{n-2}y_{n-2} + \nu_{n-1}\mathrm{Id}] \circ x_{n-2}\dots x_k(w)\\
&= \varphi x_{n-2}y_{n-2}x_{n-2}\dots x_k(w) + \nu_{n-1} v\\
&= \varphi x_{n-2}\circ [x_{n-3}y_{n-3} + \nu_{n-2}\mathrm{Id}]\circ x_{n-3}\dots x_k(w) + \nu_{n-1} v\\
&= \varphi x_{n-2} x_{n-3}y_{n-3}x_{n-3}\dots x_k(w) + \nu_{n-1}v + \nu_{n-2}v\\
&\,\,\,\vdots\\
&= \underbrace{\varphi x_{n-2} \dots x_kx_{k-1}y_{k-1}(w)}_{\in F_{k-1}} + \nu_{n-1}v + \dots + \nu_k v,
\end{align*}
so that $\varphi\psi$ preserves $F_\bullet$ and $\varphi\psi|_{F_k/F_{k-1}} = \nu_{n-1} + \dots + \nu_k$.
\end{proof}

\begin{prop}
For any $\nu \in Z(\mathfrak g_\stem^*)$ and $\gamma\in Z(\mathfrak g_\bouq^*)$, we have
\[
\mu^{-1}(\nu,\gamma)^{\theta\textup{-s}}/\GG_\stem = \left\{(x,F_\bullet) \in \wtg\colon \begin{array}{c} x|_{F_k/F_{k-1}} = \lambda_k\\ \diag(x) = \delta \\ x\textup{ satisfies~\eqref{eqn:key-stability}}\end{array}\right\}
\]
where the action of $\GG_\bouq\cong (\CC^\times)^{n-1}$ is given by
\begin{equation}
\label{eqn:bouq-action}
(t_1, \dots, t_{n-1}) \cdot (x,F_\bullet) = (\Ad_{t'}(x), t'\cdot F_\bullet), \qquad t' \colonequals \diag(t_1, \dots, t_{n-1}, 1),
\end{equation}
and where $\lambda$ and $\delta$ are given by
\begin{align}
\label{eqn:lambda}
\lambda_i &= \sum_{j=i}^{n-1}\nu_j - \frac1n(\nu_1 + 2\nu_2 + \dots + (n-1)\nu_{n-1}), \qquad\qquad i\in \{1,\dots,n\}\\
\label{eqn:delta}
\delta_i &= \gamma_i - \frac1n(\nu_1 + 2\nu_2 + \dots + (n-1)\nu_{n-1}), \qquad\qquad i\in\{1,\dots,n\}.
\end{align}
(Above, we set $\sum_{j=n}^{n-1}\nu_j\colonequals 0$ and $\gamma_n\colonequals 0$.) In particular,
\[
\mathfrak M_{(\nu,\gamma),\theta}(\mathbf v, \mathbf e_{s_{n-1}}) = \left\{(x,F_\bullet) \in \wtg\colon \begin{array}{c} x|_{F_k/F_{k-1}} = \lambda_k\\ \diag(x) = \delta \\ x\textup{ satisfies~\eqref{eqn:key-stability}}\end{array}\right\}/\GG_\bouq.
\]
\end{prop}
\begin{proof}
Lemmas~\ref{lem:theta-stable-description} and~\ref{lem:tilde-g} imply that
\[
\mu^{-1}(Z(\mathfrak g_{\mathbf v}^*))^{\theta\textup{-s}}/\GG_{\stem}\xrightarrow\sim \wtg^{\theta} \colonequals \{(x,F_\bullet)\in\wtg\colon x \textup{ satisfies~\eqref{eqn:key-stability}}\}.
\]
Equation~\eqref{eqn:M-eval} implies that $\mu_\stem$ and $\mu_\bouq$ are given by the formulas
\begin{align*}
\mu_\stem\colon \wtg^\theta&\to Z(\mathfrak g_\stem^*)\\
(x,F_\bullet)&\mapsto (\nu_1, \dots, \nu_{n-1}), \qquad \nu_i =  x|_{F_i/F_{i-1}} - x|_{F_{i+1}/F_i}\\
\mu_\bouq\colon \wtg^\theta&\to Z(\mathfrak g_\bouq^*)\\
(x,F_\bullet)&\mapsto (\gamma_1, \dots, \gamma_{n-1}), \qquad \gamma_i = x_{ii} + \frac1n(\nu_1 + 2\nu_2 + \dots + (n-1)\nu_{n-1}).
\end{align*}
The action of $\GG_\bouq \cong (\CC^\times)^{n-1}$ on $\wtg^\theta$ is given by the formula~\eqref{eqn:bouq-action}.
\end{proof}
\begin{prop}
\label{prop:generic-deformation}
For a generic $(\nu;\gamma)$, the condition~\eqref{eqn:key-stability} automatically holds, that is,
\[
\left\{(x,F_\bullet) \in \wtg\colon \begin{array}{c} x|_{F_\bullet/F_{\bullet-1}} = \lambda\\ \diag(x) = \delta \\ x\textup{ satisfies~\eqref{eqn:key-stability}}\end{array}\right\} = \underbrace{\left\{(x,F_\bullet) \in \wtg\colon \begin{array}{c} x|_{F_\bullet/F_{\bullet-1}} = \lambda\\ \diag(x) = \delta\end{array}\right\}}_{=Y_{\lambda,\delta}},
\]
where $\lambda$ and $\delta$ are given by~\eqref{eqn:lambda} and~\eqref{eqn:delta}. In particular,
\[
\mathfrak M_{(\nu,\gamma),\theta}(\mathbf v, \mathbf e_{s_{n-1}}) = Y_{\lambda,\delta}/\GG_\bouq,
\]
where the action of $\GG_\bouq$ is given by~\eqref{eqn:bouq-action}.
\end{prop}

\begin{proof}
As $\lambda$ is generic, $x$ is semisimple and there are eigenvectors $v_k$ so that $x(v_k) = \lambda_k v_k$. If $x$ preserves $\CC^S\subseteq \CC^{n-1}$, then there exist $|S|$ many eigenvectors $v_{i_1}, \dots, v_{i_{|S|}}$ of $x$ in $\CC^S$. It follows that
\[
\sum_{s\in S}\delta_s = \tr(x|_{\CC^S}) = \sum_{j=1}^{|S|} \lambda_{i_j},
\]
contradicting genericity of $(\nu,\gamma)$.
\end{proof}
\begin{cor}
\label{cor:generic-deformation}
For sufficiently generic $(\nu;\gamma)$, the variety $\wtnssst$ is diffeomorphic to $Y_{\lambda,\delta}/\GG_\bouq$.
\end{cor}
\begin{proof}
Combine Propositions~\ref{prop:diffeo-type} and~\ref{prop:generic-deformation}.
\end{proof}
\begin{rem}
The referee kindly points out that, although the results in this section are proven using the machinery of quiver varieties, it is possible to formulate the key results in a quiver-free and type-agnostic manner: We consider the $C^\infty$-trivial family $\wtg\ssslash T\to\mathfrak h$ and deform the zero-fiber $\wtnssst$ to a generic fiber $Y_{\lambda,\delta}/T = (T\backslash G)\ssslash T$, and it is natural to speculate that the cohomology of $(T\backslash G)\ssslash T$ is isomorphic to $\CC[\overline{T^*(G^\vee/U^\vee)}^{T^\vee\times T^\vee}]$, where $\bullet^\vee$ denotes Langlands dual.
\end{rem}

\section{{Cohomology of }$\wtnssst$}
\label{sec:cohomology}
The goal of this section is to prove Theorem~\ref{thm:surjectivity-and-kernel}. We construct vector bundles $\mathcal E_{S,T}$ on $\wtg$ whose Euler class is $f_{S,T}$ (Lemma~\ref{lem:eu(est)-is-fst}), along with a section whose zero locus is djsoint from $Y_{\lambda,\delta}$ (Lemma~\ref{lem:Z-Y-disjoint}) and has maximal codimension (Lemma~\ref{lem:codim-estimate}).\\

For a generic $\lambda = (\lambda_1, \dots, \lambda_n)$, write $\mathcal O_\lambda\colonequals \{(x,F_\bullet)\in \wtg\colon x|_{F_k/F_{k-1}} = \lambda_k\mathrm{Id}|_{F_k/F_{k-1}}\}$. By forgetting the flag, the subvariety $\mathcal O_\lambda$ of $\wtg$ can be identified with a regular semisimple orbit in $\mathfrak g$; in particular the group $\rmGL_n$ acts transitively on $\mathcal O_\lambda$ and the stabilizer is a maximal torus.

For $k \in [n]$, let $\mathcal F_k$ denote the tautological bundle over the flag variety $\mathcal B$ where the fiber over $F_\bullet \in \mathcal B$ is the vector space $F_k$. The bundle $\mathcal F_k$ has a natural $T$-equivariant structure. Also let $\chi_k$ denote the trivial bundle $\mathcal B\times\CC$, endowed with the action of $T$ given by $t\cdot(F_\bullet,z) = (tF_\bullet, \alpha_k(t)z)$, where $\alpha_k\colon T\to\CC^\times$ is the character $t = (t_1, \dots, t_n)\mapsto t_k$.

It is well-known (see e.g.\ \cite[Ex 3.1.2, Prop 4.4.1]{ecag}) that there is an isomorphism
\begin{align*}
\CC[\hhwh]&\xrightarrow\sim H_T^*(\mathcal B)\\
x_k&\mapsto c_1^T((\mathcal F_k/\mathcal F_{k-1})^\vee)\\
y_k&\mapsto c_1^T(\chi_k).
\end{align*}
As $\mathcal O_\lambda$ is a retract of $\wtg$, there is an isomorphism
\begin{align*}
\label{eqn:isom-O}
\CC[\hhwh]&\xrightarrow\sim H_T^*(\mathcal O_\lambda),\tag{$\diamondsuit$}\\
x_k&\mapsto c_1^T(i^*p^*((\mathcal F_k/\mathcal F_{k-1})^\vee))\\
y_k&\mapsto c_1^T(i^*p^*(\chi_k)),
\end{align*}
where $p\colon \wtg\to\mathcal B$ and $i\colon \mathcal O_\lambda\to\wtg$ denote the vector bundle map and inclusion respectively.

\begin{lem}
\label{lem:eu(est)-is-fst}
Under the isomorphism~\eqref{eqn:isom-O}, the equivariant Euler class of the bundle
\[
\mathcal E_{S,T}\colonequals i^*p^*\left(\left(\bigoplus_{s\in S}(\mathcal F_s/\mathcal F_{s-1})^\vee\right)\otimes\left(\bigoplus_{t\in T}\chi_t^\vee\right)\right)
\]
corresponds to the polynomial $f_{S,T}(\mathbf x, \mathbf y) \in \CC[\hhwh]\cong H_T^*(\mathcal O_\lambda)$.
\end{lem}
\begin{proof}
The bundle $\mathcal E_{S,T}$ is a direct sum of the line bundles
\[
\mathcal E_{\{s\},\{t\}} = i^*p^*((\mathcal F_s/\mathcal F_{s-1})^\vee\otimes\chi_t^\vee), \qquad s\in S, t\in T,
\]
and the equivariant Euler class of $\mathcal E_{\{s\},\{t\}}$ is identified with $x_s - y_t$ under~\eqref{eqn:isom-O}.
\end{proof}
For an integer $s$ and a point $(x,F_\bullet)\in\mathcal O_\lambda$, write
\[
M_x^{(s)}\colonequals (x - \lambda_1\cdot \mathrm{Id})\circ\dots\circ(x-\lambda_{s-1}\cdot \mathrm{Id}).
\]
Also write $z_1, \dots, z_n$ for the coordinate functions on $\CC^n$.
\begin{lem}
There is a $T$-equivariant section of $\mathcal E_{\{s\},\{t\}} \to \mathcal O_\lambda$ given by
\begin{align*}
\varphi_{\{s\},\{t\}}\colon \mathcal O_\lambda &\to \mathcal E_{\{s\},\{t\}}\\
(x,F_\bullet)&\mapsto((x,F_\bullet),(z_t \circ M_x^{(s)})|_{F_s}\otimes 1)
\end{align*}
\end{lem}
\begin{proof}
Observe that the linear map $M_x^{(s)}$ preserves the flag $F_\bullet$ and acts by the zero map on each $F_k/F_{k-1}$ for $1\leq k\leq s-1$. In particular $M_x^{(s)}$ acts by the zero matrix on $F_{s-1}$; thus, the function
\[
(z_t \circ M_x^{(s)})|_{F_s}\colon F_s\to\CC
\]
vanishes on $F_{s-1}$. Hence $\varphi$ is a section of $\mathcal E_{\{s\},\{t\}}\to\mathcal O_\lambda$.

Furthermore, for any $a = \diag(a_1, \dots, a_n) \in T$, we compute that
\begin{align*}
(z_t\circ M_{axa^{-1}}^{(s)})|_{aF_s}\otimes 1 &= (a_tz_t \circ M_x^{(s)})|_{aF_s} \otimes 1 \\&= (z_t \circ M_x^{(s)})|_{aF_s}\otimes a_t;
\end{align*}
hence $\varphi$ is equivariant.
\end{proof}
\begin{lem}
\label{lem:zero-locus-description}
The zero locus of
\[
\varphi_{S,T}\colonequals \bigoplus_{\substack{s\in S\\t\in T}}\varphi_{\{s\},\{t\}} \colon \mathcal O_\lambda \to \mathcal E_{S,T}
\]
is the set
\[
Z_{S,T} \colonequals \{(x,F_\bullet) \in \mathcal O_\lambda\colon v_s \in \CC^{[n]\setminus T} \textup{ for all } s \in S\}, \qquad\textup{ where } x(v_s) = \lambda_sv_s.
\]
For subsets $S,T\subset[n]$ with $|S| + |T| = n$, every $(x,F_\bullet) \in Z_{S,T}$ satisfies $x(\CC^{[n]\setminus T}) = \CC^{[n]\setminus T}$.
\end{lem}
\begin{proof}
The zero locus of $\varphi_{S,T}$ consists of points $(x,F_\bullet)\in O_\lambda$ such that $(z_t\circ M_x^{(s)})|_{F_s} \colon F_s \to \CC$ is the zero function for all $s\in S$ and $t\in T$. This is equivalent to the condition that $v_s \in \CC^{[n]\setminus T}$ for all $s \in S$.

When $|S| = n - |T|$, the eigenvectors $v_s$ of $x$ span $\CC^{[n]\setminus T}$. It follows that every $(x,F_\bullet)\in Z_{S,T}$ preserves the coordinate subspace $\CC^{[n]\setminus T}$.
\end{proof}
\begin{lem}
\label{lem:Z-Y-disjoint}
For subsets $S,T\subset[n]$ with $|S| + |T| = n$ and generic $(\lambda,\delta)$, the varieties $Z_{S,T}$ and $Y_{\lambda,\delta}$ are disjoint.
\end{lem}
\begin{proof}
Recall that $Y_{\lambda,\delta} = \{(x,F_\bullet) \in \mathcal O_\lambda\colon \diag(x) = \delta\}$. Lemma~\ref{lem:zero-locus-description} asserts that any $(x,F_\bullet) \in Z_{S,T}$ satisfies $x(\CC^{[n]\setminus T}) = \CC^{[n]\setminus T}$, so that
\[
\sum_{i \in [n]\setminus T}\delta_i = \tr(x|_{\CC^{[n]\setminus T}}) = \sum_{s\in S}\lambda_s,
\]
contradicting the genericity of $(\lambda,\delta)$.
\end{proof}
\begin{lem}
\label{lem:codim-estimate}
For subsets $S,T\subset[n]$ with $|S| + |T| = n$, we have $\codim_{\mathcal O_\lambda}(Z_{S,T}) = \rk(\mathcal E_{S,T}) = |S|\cdot|T|$.
\end{lem}
\begin{proof}
The group $\rmGL_n$ acts transitively on $\mathcal O_\lambda$ by conjugation, and the stabilizer of any point is a maximal torus. The subspace $Z_{S,T}\subseteq\mathcal O_\lambda$ is an orbit of the subgroup 
\begin{align*}
G &\colonequals \{g\in \rmGL_n\colon g(\CC^{[n]\setminus T}) = \CC^{[n]\setminus T}\}\\&=\{(g_{ij}) \in \rmGL_n\colon g_{ab} = 0\textup{ for all $a \in [n]\setminus T$ and $b \in T$}\}
\end{align*}
and in particular the stabilizer of any point is an $n$-dimensional torus. Rearranging the equality $\dim(\rmGL_n) - \dim(\mathcal O_\lambda) = \dim(G) - \dim(Z_{S,T})$ gives
\begin{align*}
\codim_{\mathcal O_\lambda}(Z_{S,T}) &= \codim_{\rmGL_n}(G) \\&=(n-|T|)|T|\\&=|S|\cdot|T|.\qedhere
\end{align*}
\end{proof}
\newtheorem*{thm:surjectivity-and-kernel}{Theorem~\ref{thm:surjectivity-and-kernel}}
\begin{thm:surjectivity-and-kernel}
The inclusion $Y\hookrightarrow\wtg$ induces a surjection
\[
\CC[\hhwh]\cong H_T^*(\wtg) \twoheadrightarrow H_T^*(Y_{\lambda,\delta})\cong H^*(\wtnssst)
\]
and the kernel contains the functions $\{f_{S,T}\colon |S| + |T| = n\}$.
\end{thm:surjectivity-and-kernel}
\begin{proof}[Proof of Theorem~\ref{thm:surjectivity-and-kernel}]
We first prove surjectivity. There are inclusions 
\[
\mu^{-1}(\lambda,\delta)^{\theta\textup{-s}} \hookrightarrow \{(\mathbf x, \mathbf y, \alpha,\beta)\in\mu^{-1}(Z(\mathfrak g_{\mathbf v})^*)\colon x_i, \varphi\textup{ injective}\} \hookrightarrow\mu^{-1}(Z(\mathfrak g_{\mathbf v}^*)).
\]
Since $\mu^{-1}(\lambda,\delta)^{\theta\textup{-s}}/\GG_{\mathbf v}$ is a smooth Nakajima quiver variety, the composite map induced on cohomology
\[
H_{\GG_{\mathbf v}}^*(\textup{pt}) \cong H_{\GG_{\mathbf v}}^*(\mu^{-1}(Z(\mathfrak g_{\mathbf v}^*))) \to H_{\GG_{\mathbf v}}^*\left( \{(\mathbf x, \mathbf y, \alpha,\beta)\in\mu^{-1}(Z(\mathfrak g_{\mathbf v})^*)\colon x_i, \varphi\textup{ injective}\}\right) \xrightarrow\eta H_{\GG_{\mathbf v}}^*(\mu^{-1}(\lambda,\delta)^{\theta\textup{-s}})
\]
is surjective by Kirwan surjectivity (Theorem~\ref{thm:kirwan-surjectivity}); in particular, the map $\eta$ is surjective. Lemma~\ref{lem:tilde-g} implies that the map $\eta$ descends to the map
\[
\eta\colon H_{\GG_\bouq}^*(\wtg) \twoheadrightarrow H_{\GG_\bouq}^*(Y_{\lambda,\delta})
\]
induced by the inclusion $Y_{\lambda,\delta}\hookrightarrow \wtg$, where $\GG_\bouq$ acts as the torus of diagonal matrices whose bottom right entry is $1$. 

The composite map $\GG_\bouq \to T_{\rmGL} \to T_{\rmPGL}$ is an isomorphism, and it follows that $\eta$ is a surjection $H_{T_{\rmPGL}}^*(\wtg) \twoheadrightarrow H_{T_{\rmPGL}}^*(Y_{\lambda,\delta})$. Corollary~\ref{cor:generic-deformation} guarantees that $H_{T_{\rmPGL}}^*(Y_{\lambda,\delta}) \cong H^*(\wtnssst)$.

We now show that the kernel contains the functions $f_{S,T}$. As the quotient $T = T_{\mathrm{SL}}\to T_{\rmPGL}$ induces isomorphisms $H_T^*(\wtg) \cong H_{T_{\rmPGL}}^*(\wtg)$ and $H_T^*(Y_{\lambda,\delta})\cong H_{T_{\rmPGL}}^*(Y_{\lambda,\delta})$ it suffices to show that $f_{S,T}$ vanishes under $H_T^*(\wtg)\twoheadrightarrow H_T^*(Y_{\lambda,\delta})$. Lemma~\ref{lem:eu(est)-is-fst} asserts that $\eu^T(\mathcal E_{S,T}) = f_{S,T}$. As the zero locus $Z_{S,T}$ of our section $\varphi_{S,T}$ of $\mathcal E_{S,T}$ has maximum possible codimension (Lemma~\ref{lem:codim-estimate}), we have $\eu^T(\mathcal E_{S,T}) = [Z_{S,T}]^T$ in $H_T^*(\mathcal O_\lambda)$ (\cite[\S 2.3]{ecag}). Finally, fundamental classes $[V]^T\in H_T^*(X)$ vanish under the restriction $H_T^*(X) \to H_T^*(X\setminus V)$ (\cite[pg.\ 398]{ecag}), and the varieties $Z$ and $Y_{\lambda,\delta}$ are disjoint (Lemma~\ref{lem:Z-Y-disjoint}); in particular $f_{S,T} = [Z_{S,T}]^T$ vanishes under $H_T^*(\mathcal O_\lambda)\to H_T^*(Y_{\lambda,\delta})$. Since the inclusion $\mathcal O_\lambda\to\wtg$ induces an isomorphism in cohomology, the result follows.
\end{proof}
\begin{rem}
Theorem~\ref{thm:surjectivity-and-kernel} implies that $\wtnssst$ has vanishing odd cohomology. (The vanishing of odd cohomology is known for all smooth quiver varieties \cite[Thm 7.3.5]{nakajima01}.)

Because $\overline{T^*(G/U)}^{T\times B/U}$ is a zero-dimensional (non-reduced) scheme, its coordinate ring is also finite-dimensional as a $\CC$-algebra.

As Theorem~\ref{thm:surjectivity-and-kernel} gives a surjective map between two finite-dimensional $\CC$-algebras, the Hikita conjecture in fact predicts that the map in Theorem~\ref{thm:surjectivity-and-kernel} is an isomorphism. 
\end{rem}

\appendix
\section{Other examples of Conjecture~\ref{conj:quot-unquot}}
\label{appendix:examples}

Let $e\in\mathcal N$ be a nilpotent matrix with Jordan type given by a partition $\lambda$, and let $e^\vee$ be a nilpotent matrix with Jordan type given by the transpose partition $\lambda'$. The closure $X$ of the nilpotent orbit containing $e$ is a symplectic singularity and its symplectic dual $X^!$ is the Slodowy slice in $\mathcal N$ through $e^\vee$ (\cite{hikita17}). This behavior is an instance of the ``matching of strata'' phenomenon in \cite[\S 5.3]{kamnitzer22}. Below, we give examples of Conjecture~\ref{conj:quot-unquot} arising from the duality between nilpotent orbits and Slodowy slices.

\begin{example}
Let $X = \overline{\OO_{\min}}$ be the minimal nilpotent orbit closure in type $A_n$. Then $X^!$ is the Slodowy slice through a subregular nilpotent matrix; in particular $X^! \cong \CC^2/A_n$ can be expressed as the Nakajima quiver variety $\mathfrak M_{0,0}(\mathbf 1,\mathbf 0)$ for the affine $\widetilde A_n$ Dynkin quiver.

The coordinate ring of the Hamiltonian reduction $X\ssslash T \colonequals \{x\in \overline{\OO_{\min}}\colon \diag(x) = 0\}\sslash T$ is isomorphic to $\CC$; thus $X\ssslash T$ is a point. On the other hand, $X^{!,\uparrow} \cong T^*\CC^n$ is smooth and the fixed point subscheme $(X^{!,\uparrow})^{T^{!,\uparrow}}$ is one reduced point.

In this case, the Hikita conjecture
\[
H^*\left(\widetilde{X\ssslash T}\right) \cong \CC\left[(X^{!,\uparrow})^{T^{!,\uparrow}}\right]
\]
holds as both sides are isomorphic to $\CC$.
\end{example}
\begin{example}
\label{ex:slodowy}
Let $X$ be the Slodowy slice through a nilpotent matrix whose Jordan type has one block of size $n-2$ and another block of size $2$, so that $X^!$ is the nilpotent orbit closure
\[
X^! = \{x\in\mathcal N_{\mathfrak{sl}_n}\colon x^2 = 0, \dim\img(x) = 2\}.
\]
The variety $X^!$ has an interpretation as the Nakajima quiver variety $\mathfrak M_{0,0}((2),(n))$ associated to the quiver with one vertex and no edges \cite{kp79}.

A result of Maffei \cite{maffei05} guarantees that $X$ is the Nakajima quiver variety $\mathfrak M_{0,0}(\mathbf v, \mathbf w)$ for the $A_{n-1}$ Dynkin quiver, where $\mathbf v = (1,2, 2, \dots, 2, 2, 1)$ and $\mathbf w = (0,1,0,\dots,0,1,0)$. The torus $T$ acting on $X$ can be identified with the residual action of $\prod_{i\in I}\rmGL(w_i)$; in particular $X\ssslash T$ is the Nakajima quiver variety $\mathfrak M_{0,0}(\mathbf v', \mathbf 0)$ for the affine $\widetilde D_n$ Dynkin quiver with $\mathbf v'$ equal to the minimal imaginary root (see Figure~\ref{fig:xssslasht-slodowy}). In particular $X\ssslash T$ is the Kleinian singularity $\CC^2/D_n$.
\begin{figure}[ht]
\centering\includegraphics{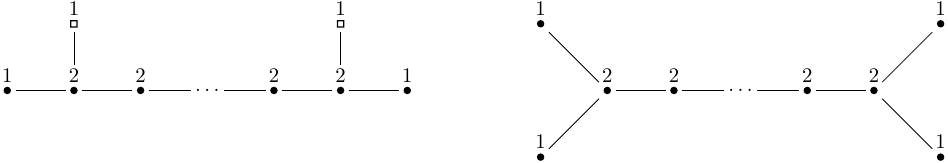}
\caption{Left: The $A_{n-1}$ Dynkin quiver defining the Slodowy slice $X$. Right: The affine $\widetilde D_n$ Dynkin quiver, with dimension vector $\mathbf v'$, defining $X\ssslash T = \CC^2/D_n$.} 
\label{fig:xssslasht-slodowy}
\end{figure}

On the other hand, $X^{!,\uparrow} = T^*\Hom(\CC^2,\CC^n)\ssslash \rmSL_2$ is the minimal nilpotent orbit closure in $\mathfrak{so}_{2n}$ \cite[Prop 3.18]{jia21}. In this case, the Hikita conjecture for $X\ssslash T$ and $X^{!,\uparrow}$ was verified in \cite{shlykov24}.
\end{example}
\begin{rem}
Now swap the roles of $X$ and $X^!$ in Example~\ref{ex:slodowy}: Let $X$ be the nilpotent orbit closure
\[
X = \{x\in\mathcal N_{\mathfrak{sl}_n}\colon x^2 = 0, \dim\img(x) = 2\}.
\]
Then $X^!$ is the Nakajima quiver variety $\mathfrak M_{0,0}(\mathbf v, \mathbf w)$ for the type $A_{n-1}$ quiver, where $\mathbf v = (1,2, 2, \dots, 2, 2, 1)$ and $\mathbf w = (0,1,0,\dots,0,1,0)$ and $X$ is the Nakajima quiver variety $\mathfrak M_{0,0}((2),(n))$ associated to the quiver with one vertex and no edges.

In this case, the Hamiltonian reduction $X\ssslash T$ is the Nakajima quiver variety $\mathfrak M_{0,0}(\mathbf v, \mathbf 0)$ associated to the star shaped quiver $Q_n^{\mathrm{star}}$ with vertices $\{*\}\sqcup\{1,2,\dots,n\}$, edges $\{\{*,i\}\colon i\in[n]\}$, and dimension vector $\mathbf v'' = (2; 1,\dots,1)$ (see Figure~\ref{fig:star-quiver}).

\begin{figure}[ht]
\centering\includegraphics{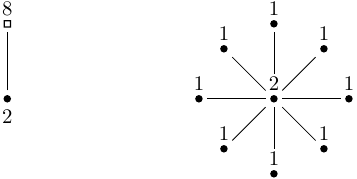}
\caption{Left: The quiver defining the nilpotent orbit $X$ in the case $n = 8$. Right: The star quiver $Q_n^{\mathrm{star}}$, with dimension vector $\mathbf v''$, defining $X\ssslash T$, also for $n = 8$.}
\label{fig:star-quiver}
\end{figure}

By \cite[Thm 5.1, Prop 5.20]{bfn-ring} (see also \cite[(4.58)]{dg19}), the corresponding BFN Coulomb branch $\mathcal M_C$ associated to $Q_n^{\mathrm{star}}$ can be identified with the Hamiltonian reduction $\mathbb M(Q,\mathbf v, \mathbf w)\ssslash \prod_{i\in I}\rmSL(v_i)$, where $Q$ is the type $A_{n-3}$ Dynkin quiver, $\mathbf v = (2, 2, \dots, 2, 2)$, and $\mathbf w = (2, 0,\dots,0,2)$. In particular, $X\ssslash T$ is conjecturally symplectic dual to $\mathcal M_C = X^{!,\uparrow}$ (see Figure~\ref{fig:ginzburg-kazhdan-wb}).
\begin{figure}[ht]
\centering\includegraphics{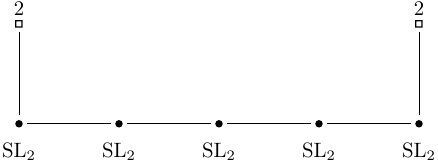}
\caption{The variety $\mathcal M_C$, for $n = 8$, as a Hamiltonian reduction of the form $(T^*\Hom(\CC^2,\CC^2))^6\ssslash (\rmSL_2)^5$.}
\label{fig:ginzburg-kazhdan-wb}
\end{figure}

The variety $\widetilde{X\ssslash T} = \mathfrak M_{\theta,0}(\mathbf v'', \mathbf 0)$ has appeared in the literature as the \emph{hyperpolygon space} and the cohomology is known \cite[Thm 7.1]{konno02} (see also \cite[Thm 3.1, Thm 3.2]{hp05}) to be
\[
H^*(\widetilde{X\ssslash T}) = \CC[z_1, \dots, z_n, p]/I,
\]
where $I$ is the ideal generated by all elements $p - z_i^2$ as well as all monomials of degree $2(n-2)$, where $\deg(z_i)\colonequals 2$ and $\deg(p)\colonequals 4$. The Hikita conjecture predicts that this ring is also the coordinate ring of $(X^{!,\uparrow})^{T^{!,\uparrow}}$. To the best of our knowledge, this case of the Hikita conjecture is still open.
\end{rem}

\section{Bellamy--Schedler criterion}
Recall that the Ringel form on $\ZZ^{Q_0}$ is defined by
\[
\langle \alpha,\beta\rangle \colonequals \sum_{i\in Q_0}\alpha_i\beta_i - \sum_{\alpha\in Q_1}\alpha_{t(a)}\beta_{h(a)}.
\]
The corresponding Euler form is defined by
\begin{align*}
(\alpha,\beta) &\colonequals \langle \alpha,\beta\rangle + \langle \beta,\alpha\rangle
\\&=\label{eqn:euler} 2\sum_{i\in Q_0}\alpha_i\beta_i - \sum_{a \in Q_1}(\alpha_{t(a)}\beta_{h(a)} + \beta_{t(a)}\alpha_{h(a)}).\tag{$\star$}
\end{align*}

We use a criterion by Crawley-Boevey to determine when $\mathbf v$ is in $\Sigma_0$:
\begin{prop}[{\cite[Cor 5.7]{cb01}}]
\label{prop:cb-criterion}
If $\mathbf v \in \NN^I$ then $\mathbf v \in \Sigma_0$ if and only if $\mathbf v > 0$ and $(\mathbf w,\mathbf v-\mathbf w)\leq -2$ whenever $\mathbf w\in \NN^I$ and $0<\mathbf w<\mathbf v$. 
\end{prop}
Henceforth, the quiver $Q$ is fixed to be the augmented bouquet quiver $Q_n^+$.
\begin{lem}
\label{lem:v-in-sigma}
The dimension vector $\hat{\mathbf v} = (1,2,\dots,n-1; 1,1,\dots,1)$ is in $\Sigma_0$.
\end{lem}
\begin{proof}
We use Proposition~\ref{prop:cb-criterion}. Write $\mathbf w= (w_{s_1}, \dots, w_{s_{n-1}}; w_{b_1}, \dots, w_{b_n})$, so that
\[
\hat{\mathbf v} - \mathbf w = (1 - w_{s_1}, 2 - w_{s_2}, \dots, (n-1)-w_{s_{n-1}};1-w_{b_1},\dots, 1-w_{b_n}).
\]
Replacing $\mathbf w$ with $\hat{\mathbf v} - \mathbf w$ if necessary, we may assume that $w_{s_{n-1}} \leq (n-1)-w_{s_{n-1}}$.
Equation~\eqref{eqn:euler} reads
\begin{align*}
(\mathbf w,\hat{\mathbf v}-\mathbf w) &= 2\sum_{i=1}^{n-1}w_{s_i}(i-w_{s_i}) + 2\sum_{i=1}^n w_{b_i}(1-w_{b_i}) \\&\qquad\quad -\sum_{i=1}^{n-2}\! \Big(w_{s_i}(i+1-w_{s_{i+1}})+ w_{s_{i+1}}(i - w_{s_i})\Big) - \sum_{i=1}^n\! \Big(w_{s_{n-1}}(1-w_{b_i}) + (n-1-w_{s_{n-1}})w_{b_i}\Big)
\end{align*}
Since $w_{b_i} \in \{0,1\}$ we know $w_{b_i}(1-w_{b_i}) = 0$. Furthermore, since $w_{s_{n-1}} \leq n-1-w_{s_{n-1}}$ the quantity $(\mathbf w, \hat{\mathbf v} - \mathbf w)$ maximized when $1 - w_{b_i} = 1$ for all $i$. We conclude
\begin{align*}
(\mathbf w, \mathbf v-\mathbf w)&\leq 2\sum_{i=1}^{n-1}w_{s_i}(i-w_{s_i}) -\sum_{i=1}^{n-2}\!\Big(w_{s_i}(i+1-w_{s_{i+1}})+ w_{s_{i+1}}(i - w_{s_i})\Big) - nw_{s_{n-1}}
\\&= 2\sum_{i=2}^{n-1}w_{s_i}(i-w_{s_i}) -\sum_{i=1}^{n-2}\!\Big(w_{s_i}(i+1-w_{s_{i+1}})+ w_{s_{i+1}}(i - w_{s_i})\Big) - nw_{s_{n-1}}
\\&=2(-w_{s_2}^2 - \dots - w_{s_{n-1}}^2) + 2(w_{s_1}w_{s_2} + w_{s_2}w_{s_3} + \dots + w_{s_{n-2}}w_{s_{n-1}}) - 2w_{s_1}
\\&=2(-w_{s_1}^2 - w_{s_2}^2 - \dots - w_{s_{n-1}}^2) + 2(w_{s_1}w_{s_2} + w_{s_2}w_{s_3} + \dots + w_{s_{n-2}}w_{s_{n-1}}),
\end{align*}
where the first and last equality follow from the fact that $w_{s_1} \in \{0,1\}$.

Let $m$ and $M$ be the minimum and maximum index so that $w_{s_m}, w_{s_M}\neq 0$. Then
\[
2(-w_{s_1}^2 - w_{s_2}^2 - \dots - w_{s_{n-1}}^2) + 2(w_{s_1}w_{s_2} + w_{s_2}w_{s_3} + \dots + w_{s_{n-2}}w_{s_{n-1}})
\]
is equal to
\[
2(-w_{s_m}^2 - w_{s_{m+1}}^2 - \dots - w_{s_M}^2) + 2(w_{s_m}w_{s_{m+1}} + \dots + w_{s_{M-1}}w_{s_M}).
\]
The AM-GM inequality implies that
\[
-w_{s_m}^2 - 2w_{s_{m+1}}^2 - \dots - 2w_{s_{M-1}}^2 - w_{s_M}^2 + 2w_{s_m}w_{s_{m+1}} + \dots + 2w_{s_{M-1}}w_{s_M} \leq 0.
\]
It follows that
\[
2(-w_{s_m}^2 - w_{s_{m+1}}^2 - \dots - w_{s_M}^2) + 2(w_{s_m}w_{s_{m+1}} + \dots + w_{s_{M-1}}w_{s_M}) \leq -w_{s_m}^2 - w_{s_M}^2 \leq -2
\]
since $w_{s_m}, w_{s_M}\neq 0$.
\end{proof}
Recall that the vector $\hat{\mathbf v}$ is \emph{anisotropic} if $\langle \alpha,\alpha\rangle < 0$.

\begin{lem}
\label{lem:v-anisotropic}
For $n\geq 4$, the dimension vector $\hat{\mathbf v} = (1, 2, \dots, n-1; 1, \dots, 1)$ is anisotropic.
\end{lem}
\begin{proof}
We compute that
\begin{align*}
\langle\hat{\mathbf v}, \hat{\mathbf v}\rangle &= \sum_{i\in Q_0}v_i^2 - \sum_{e \in Q_1}v_{t(e)}v_{h(e)}\\&= \left(\sum_{i=1}^{n-1}i^2\right) + n - \left(\sum_{i=1}^{n-1}i(i+1)\right)\\&= n - \frac{n(n-1)}2
\end{align*}
is negative when $n\geq 4$.
\end{proof}

\end{document}